\documentclass{amsart}

\usepackage{xypic}
\usepackage{graphicx}
\usepackage[percent]{overpic}

\newcommand{\R}{\mathbf{R}}
\newcommand{\Z}{\mathbf{Z}}
\newcommand{\N}{\mathbf{N}}
\newcommand{\F}{\mathcal{F}}
\newcommand{\eps}{\varepsilon}

\newtheorem{theorem}{Theorem}[section]
\newtheorem{lemma}[theorem]{Lemma}
\newtheorem{proposition}[theorem]{Proposition}
\newtheorem{corollary}[theorem]{Corollary}

\theoremstyle{definition}
\newtheorem{definition}[theorem]{Definition}
\newtheorem{example}[theorem]{Example}

 \theoremstyle{remark}
\newtheorem{remark}[theorem]{Remark}

\numberwithin{equation}{section}

\begin{document}

\title[Singular Riemannian Foliations]{Singular Riemannian Foliations:\\ Exceptional leaves; Tautness}

\author{Eva Nowak}
\address{Eva Nowak\\
Mathematisches Institut\\
Universit\"at zu K\"oln\\
Weyertal 86-90\\
50931 K\"oln\\
Germany}
\email{enowak@mi.uni-koeln.de}


\begin{abstract}
For a singular Riemannian foliation whose leaves are properly embedded, we show in the first part of this article the existence of global tubular neighbourhoods, and we develop a global description of the foliation as stratification by types of leaves.

The second part deals with the further restriction to a foliation without horizontal conjugate points, introduced by Lytchak and Thorbergsson, which in the special case of an isometric group action equals the concept of variationally completeness. Therefrom, we deduce a global geometric description -- the focal points of the leaves are exactly the singular points -- as well as a topological one: tautness of the leaves.
\end{abstract}

\maketitle
\section{Introduction}

It is a classical result of Bott and Samelson \cite{bottsamelson} that the orbits of a variationally complete group action are taut\footnote{In the context of foliations there exists a completely different meaning of this notion, see \cite{conlon}, but we use \emph{tautness} only in the Morse theoretical meaning as introduced in \cite{grove} as well as \cite{terngthorb2}.}. Isometric group actions are generalized to foliations via the notion of singular Riemannian foliation, see \cite{molino}; and Lytchak and Thorbergsson defined in \cite{lth} the analog of variational completeness for such foliations via the notion of foliation without horizontally conjugate points. The main aim of this paper is to show that tautness also holds in this more general situation. A basic approach is to generalize the explicit construction of linking cycles done by Bott and Samelson. Pursuing this strategy, we need and find global structural characteristics of the foliations concerned. While singular Riemannian foliations are well understood locally, their global structure can be arbitrarily complicated at least in the non compact case.

After a short introduction, we restrict our interest to properly embedded leaves from Chapter \ref{s:ex} forward. At first, we show the existence of global tubular neighborhoods, see Theorem \ref{theotube}, which is the basic tool to reveal global structural results in the following.

Then, we have a look at exceptional leaves, i.e. leaves of highest dimension with non-trivial holonomy. This kind of leaves does not occur too often as the set of specificly regular leaves is open, connected and dense in the manifold (Propositions \ref{propregdense} and \ref{propregcon}). We define the type of a leaf and show that the stratum of exceptional leaves as well as the manifold itself are stratified by the types of leaves.

In the third chapter, we analyse foliations without horizontally conjugate points (with properly embedded leaves). Locally constructed variations of horizontal geodesics can be extended globally, which leads to the geometric characterization formulated in Theorem \ref{theowhcp}: A singular Riemannian foliation has no horizontally conjugate points if and only if the singular points are exactly the focal points of the leaves. For this theorem, we need the extension of homothetical transformations, defined in \cite{molino} within tubular neighborhoods, to the negatives which is done in Proposition \ref{prophomotheties}. Even more, we show that the homothetical transformations are diffeomorphisms that respect not only the leaves but also the types of leaves. There are two other papers that among other things show the existence of negative homotheties, namely the one of Alexandrino and T\"oben, \cite{at}, and the one of Lytchak and Thorbergsson, \cite{lth07}. As these proofs are contemporaneous but independent and use completely different methods, we include our proof here.

With these structural results we finally can address ourselves to Morse theory of the leaves, i.e. the tautness problem. As a first step, we show that the leaves of a singular Riemannian foliation without horizontally conjugate points are $0$-taut. Therefrom, we conclude that there are no exceptional leaves if the foliated manifold is simply connected. At last, we show Theorem \ref{theotaut}: The leaves of a singular Riemannian foliation without horizontally conjugate points are taut.

\emph{Acknowledgements.} This paper essentially is my thesis which I wrote under the supervision of G. Thorbergsson at the University of Cologne. He deserves my gratitude for the constantly encouraging support during the last years. Furthermore, I want to thank A. Lytchak and S. Wiesendorf for essential hints towards Proposition \ref{propiso} and Theorem \ref{theotaut} in this generality.

\section{Preliminaries}

\subsection{Definitions and notation}

Let $(M, g)$ be a Riemannian manifold of dimension $n$. A decomposition $\F$ of $M$ into smooth, connected, injectively immersed submanifolds, the \emph{leaves,} is called a \emph{singular Riemannian foliation} if the following two conditions hold:
\begin{trivlist}
\item{(I)} The modul $\Xi_{\F}$ of smooth vector fields on $M$ that are tangential to the leaves is \emph{transitive,} i.e. for a leaf $N_p$ through $p$, and for every vector $v \in T_pN_p$ there is a vector field $X \in \Xi_{\F}$ such that $X(p) = v.$ Thus, the foliation is \emph{singular} in the sense of \cite{stefan} and \cite{sussmann}.
\item{(II)} The metric $g$ on $M$ is \emph{adapted} to $\F$: Each geodesic in $M$ that hits one leaf perpendicularly remains perpendiculary to all leaves it passes through. Such a geodesic is called \emph{horizontal.}
\end{trivlist}

Let $k$ denote the maximal dimension of the leaves. A leaf is \emph{regular} if its dimension is equal to $k,$ otherwise it is \emph{singular}. Equally, a point $p$ in $M$ is \emph{regular} (resp.~\emph{singular}) if the leaf $N_p$ is regular (resp.~singular).  The manifold $M$ is stratified by the dimension of the leaves. The \emph{regular stratum} $\Sigma_{\rm reg}$ is the union of all leaves of dimension $k$. The union of all leaves of lower dimension, the singular stratum, is denoted by $\Sigma_{\rm sing}.$  Each stratum $\Sigma_q$ of leaves of dimension $q$ is an embedded submanifold, $\bigcup_{q' \le q} \Sigma_{q'}$ is closed, and the regular stratum is open, connected, and dense in $M.$

For the leaf $N_p \subset M$ through $p,$ let $\nu N_p$ denote the normal bundle of $N_p$ in $TM|_{N_p}.$ If  $P \subset N_p$ is a relatively compact, open, and connected subset, $B_{\eps}(P)$ denotes the bundle of open normal balls with radius $\eps$ around $P$ in $N_p \subset \nu N_p.$ For $\eps$ sufficiently small, the normal exponential map is a diffeomorphism of $B_{\eps}(P)$ onto its image in $M,$ which is called a \emph{(distinguished) tubular neighborhood of $P$ (resp.~around $p$) with radius $\eps:$}
$$\exp^\perp: \nu N_p \to M; \exp^\perp(B_{\eps}(P))=: {\rm Tub}_\eps(P).$$
A \emph{tube} of radius $\delta < \eps$ around $P$ is the union of all points in ${\rm Tub}_\eps(P)$ whose distance to $P$ is equal to $\delta:$
$$S_\delta^P := \{x \in {\rm Tub}_\eps(P) \mid d(x, P) = \delta\}.$$
Let $\pi: {\rm Tub}_\eps(P) \to P$ denote the projection along horizontal geodesics. We call $$S_\eps(p) := \pi^{-1}(p)$$ a \emph{slice.}  If $y \in S_\eps(p),$ and $\eps$ sufficiently small, the leaf $N_y$ cuts the slice transversally. 

\noindent\begin{minipage}{0.6\textwidth}
\begin{center}
\begin{overpic}[width=7cm,
]{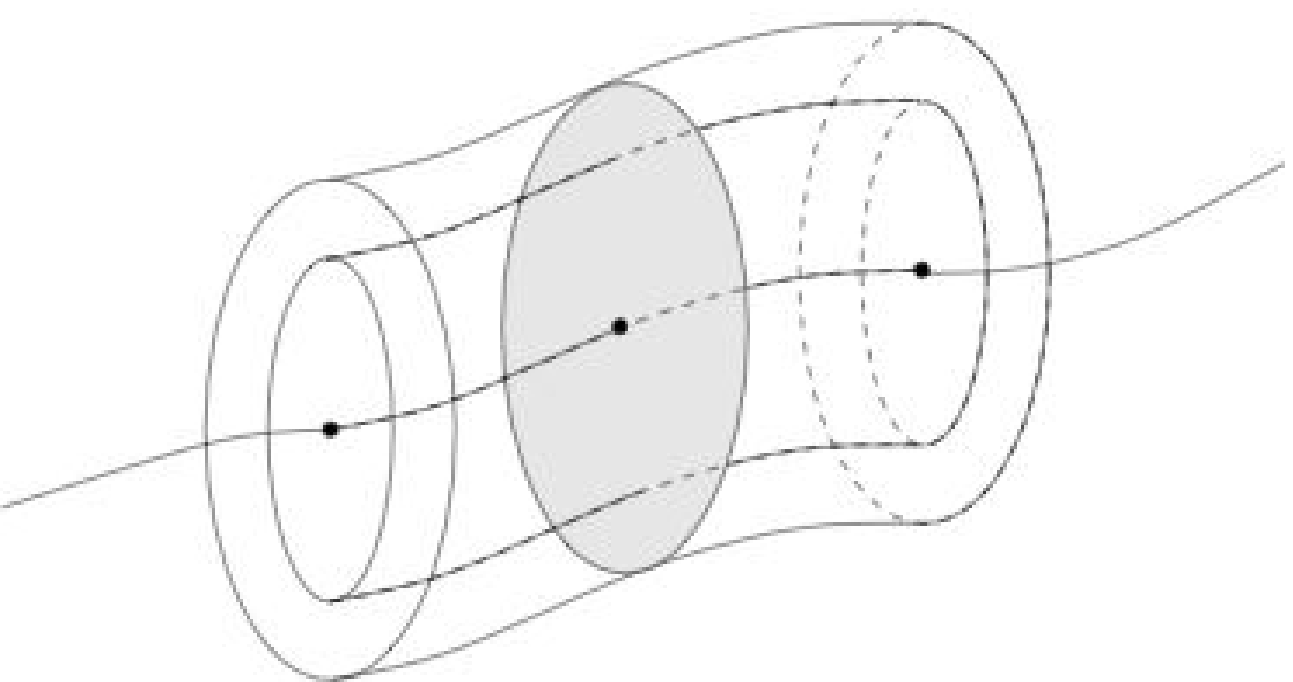}
\put(48,25){$p$}\put(95,35){$N_p$}\put(45,35){$S_\eps(p)$}\put(17,42){${\rm Tub}_\eps(P)$}\put(36,19){$P$}\put(62,15){$S_\delta^P$}
\end{overpic} 
\end{center}
\end{minipage}
\begin{minipage}{0.4\textwidth}
By the following Lemma \ref{lemlocaltube}, the connected component of \mbox{$N_y \cap {\rm Tub}_\eps(P)$} that contains $y$ lies in the tube $S_{d(y, p)}^P.$  Such a component $P_y \subset N_y$ is called a \emph{plaque}. For more details we refer to \cite{molino}, Chapter 6, wherefrom we just cite the following facts:
\end{minipage}

\begin{lemma}\label{lemlocaltube}
\begin{trivlist}
\item{\rm (i)} The projection $\pi: P_y \to P$ is a surjective submersion, and the fiber $\pi^{-1}(x)$ meets all the plaques in the tubular neighborhood for all $x \in P$.
\item{\rm (ii)} The distance between neighboring leaves is locally constant for the topology of leaves: $d(\tilde y, P) = d(y, P)$ for all $\tilde y$ in $P_y.$
\end{trivlist}
\end{lemma}

For every point in $\Sigma_{\rm reg}$ there exists an open neighborhood $U,$ called \emph{simple,} such that the projection $\pi_U : U \to \bar U$ onto the manifold of plaques is a submersion. The manifold $\bar U$ can be endowed with a Riemannian metric $g_{\bar U}$ such that the transversal part of $g$ is the pullback of $g_{\bar U}$ via $\pi,$ and 

\begin{proposition}
If $\gamma$ is a horizontal geodesic in $U$, its projection $\pi_U(\gamma)$ is a geodesic in $\bar U$.
\end{proposition}

Let us have a look at a plaque $P_N.$ By the composition of homothetical transformations in the normal bundle $\nu N$ with $\exp^\perp: B_{\eps}(P_N) \to {\rm Tub}_\eps(P_N),$ we can define a \emph{homothetical transformation} $h_\lambda$ with respect to $P_N$ for all positive $\lambda \in \R$ within the domain ${\rm Tub}_{\min (\eps, \frac{1}{\lambda}\eps)}(P_N),$ and

\begin{lemma}\label{lemposhom}
The homothetical transformation $h_\lambda$ sends plaque to plaque and therefore respects the singular foliation $\F$.
\end{lemma}

If $\gamma$ is a horizontal geodesic, it follows by this lemma that within a tubular neighborhood of $\gamma(0)$
$$\dim(N_{\gamma(t)}) = \text{ constant for all } t > 0.$$


\subsection{Sliding along the leaves}\label{ss:sliding}

If $U \subset M$ is simple, the belonging of points to their plaque in $U$ uniquely defines diffeomorphisms between local transversals. Given $x$ and $x'$ in the same plaque, we call the corresponding diffeomorphism \emph{sliding along the leaves from $x$ to $x'$,} see \cite{molino}, Section 1.7. By coverings with simple sets, this definition can easily be extended to arbitrary points in a regular leaf and corresponding local transversals that are sufficiently small.
Let $c$ denote a path from $x$ to $x'$ and let $h_c$ denote the germ of this diffeomorphism at $x$ along $c.$ Note that $h_c$ does neither depend on the choice of the covering nor on the choice of local transversals, and concerning $c$ only on its homotopy class. This germ is called \emph{sliding along the leaves from $x$ to $x'$ along $c$.}

For every point $c(t), t \in [0,1],$ the slice of minimal radius of the covering tubes $S_{\eps_{\min}}(c(t))$ is a local transversal which then defines for every $y \in S_{\eps_{\min}}(x)$ a unique, continuous path $\tilde c_y: [0,1] \to L_y$ from $y$ to its image under sliding along $c$ contained in $S_{\eps_{\min}}(x'),$ whereby $d(\tilde c_y(t), c(t)) = d(y, x)$ for all $t \in [0,1].$

Thus, in the following we sometimes speak about sliding of distances instead of points, which can be used not only within regular leaves and simple neighborhoods but -- by Lemma \ref{lemlocaltube} -- always as long as the path $c$ is completely contained in a tubular neighborhood.


\subsection{Focal points, cut points, and the cut locus of a leaf}

We briefly recall the notions above. For details see \cite{docarmo}, Section 10.4, and \cite{toeben}:

Let $N \subset M$ be a submanifold with tangent bundle \mbox{$TN \subset TM|_N,$} normal bundle $\nu N,$ and normal exponential map \mbox{$\exp^\perp: \nu N \to M.$} If $v$ is a critical point of $\exp^\perp,$ we call $v$ a \emph{focal vector} and $\exp^\perp(v)$ a \emph{focal point.}

Let $\gamma$ denote a geodesic with $\gamma(0) \in N, \dot\gamma(0) \perp N.$ An \emph{$N$-Jacobi field along $\gamma$} is a Jacobi field $J$ along $\gamma$ such that $$J(0) \in T_{\gamma(0)}N\ \text{ and }\ J'(0) + A_{\dot\gamma(0)}(J(0)) \in \nu_{\gamma(0)}N,$$ where $A_{\dot\gamma(0)}: T_{\gamma(0)}N \to T_{\gamma(0)}N;\; A_{\dot\gamma(0)}X = -(\nabla^M_X\dot\gamma(0))^{T}$ denotes the shape operator. $J$ is a $N$-Jacobi field if and only if it is the variational vector field of a variation $f_s$ of $\gamma,$ where all paths $f_s$ are geodesics with $f_s(0) \in N, \dot f_s(0) \perp N.$
\begin{proposition} The following holds for a geodesic $\gamma$ perpendicular to $N$ in $\gamma(0)$:

\begin{tabular}[t]{lcl}
$\gamma(t)$ is a focal point of $N$ & $\iff$ & $t\dot\gamma(0)$ is a critical point of $\exp^\perp$\\
& $\iff$ & there exists a non-vanishing $N$-Jacobi\\
&& field $J$ along $\gamma$ with $J(t) = 0.$
\end{tabular}
\end{proposition}

The \emph{multiplicity of $\gamma(t)$} is equal to the dimension of the kernel of $d(\exp^\perp)_{t\dot\gamma(0)}.$ This also equals the dimension of the space of Jacobi fields along $\gamma$ vanishing in $t.$

Let now $N \subset M$ be properly immersed. Let $\gamma_v$ denote the geodesic starting in the direction of $v \in TM.$ Furthermore, let $\nu^1N$ be the unit normal bundle of $N.$ Then the \emph{cut locus function} is defined by $$\sigma: \nu^1N \to [0, \infty]; \sigma(v) = \sup \{t \in \R \mid d(\gamma_v(t), N) = t\}.$$ The \emph{cut locus of $N$} is defined to be the set $$\mathcal C_N := \exp^\perp \{\sigma(v)v \mid \sigma(v) < \infty, v \in \nu^1N\}.$$

We call $v \in \nu N$ \emph{minimal}, if $\|v\| \le \sigma(\frac{v}{\|v\|})$. The associated geodesic $\gamma_v|_{[0, 1]}$ and each of its  parametrizations of constant speed are called \emph{minimal geodesics.} A minimal vector $v$ is a \emph{cut vector,} if there is a different minimal $w \in \nu N$ with the same end point; $\exp^\perp(v)$ is then called a \emph{cut point.} Then, $\|v\| = \|w\|.$ The following proposition holds:

\begin{proposition}[\cite{toeben}]\label{propcutlocus}
\begin{trivlist}
\item{\rm (i)} Each point of $\mathcal C_N$ is a cut point or a focal point.
\item{\rm (ii)} The cut locus function $\sigma$ is upper semi-continuous, and
\item{\rm (iii)} if $N$ is properly immersed, $\sigma$ is continuous and $\mathcal C_N$ is closed in $M.$
\end{trivlist}
\end{proposition}

\section{Exceptional leaves -- a stratification}\label{s:ex}

In general, the leaves of a singular Riemannian foliation are embedded submanifolds. For the following, we additionally assume the embeddings to be proper. It is easy to see that a tubular neighborhood then contains only finitely many plaques of every meeting leaf. As a counterexample take the one-dimensional, irrational foliation of a torus.

\subsection{Global tubular neighborhoods}

For a singular Riemannian foliation with properly embedded leaves, we can show the existence of global tubular neighborhoods, which will be a useful tool later:

\begin{theorem}\label{theotube}
Let $M$ be complete and $(M, \mathcal F)$ be a singular Riemannian foliation with properly embedded leaves. Then, for every leaf $N \subset M$ there is an $\varepsilon_N > 0$ such that the normal exponential map $${\rm exp}_N^{\perp}: B_{\varepsilon_N}(N) \to {\rm Tub}_{\varepsilon_N}(N)$$ is a diffeomorphism.
 \end{theorem}

 \begin{proof} The proof is devided into two steps:
\begin{trivlist}
\item{(i)} The distance of the leaves is globally constant, i.e.:  If $\gamma: [0,1] \to M$ is a horizontal geodesic that is locally minimizing between $N_0$ and $N_1,$ then for every $x \in N_0$ there is a horizontal geodesic of length $L(\gamma)$ from $x$ to $N_1,$ and this geodesic, too, is locally minimizing between $N_0$ and $N_1.$ Hereby, we call $\gamma$ \emph{locally minimizing,} if there is a neighborhood $P_0 \subset N_0$ around $\gamma(0)$ and a neighborhood $P_1 \subset N_1$ around $\gamma(1)$ such that $d(P_0, P_1) = L(\gamma).$

\item{(ii)} The distance between a leaf $N$ and its cut locus $\mathcal C_N$ is bounded away from zero. For every lower bound $r > 0,$ it follows by definition of $\mathcal C_N$ that the map $${\rm exp}_N^{\perp}: B_{r}(N) \to {\rm Tub}_{r}(N)$$ is a diffeomorphism. 
\end{trivlist}

We prove step (i) as follows: With regard to $x,$ it is clear that this property is closed. That it is open, too, does not follow directly from the existence of local tubular neighborhoods as, in general, $\gamma(0)$ is not contained in such a neighboorhood of a plaque of $N_1,$ for example if $\gamma(0)$ is singular but $N_1$ regular. So, we have to show that there is a neighborhood $U \subset P_0$ of $\gamma(0)$ such that $d(x, P_1) = L(\gamma)$ for all $x \in U.$ For this purpose take a covering of $\gamma$ by tubular neighborhoods as follows:

\noindent\begin{minipage}{0.48\textwidth}
\begin{center}
\begin{overpic}[width=6cm,
]{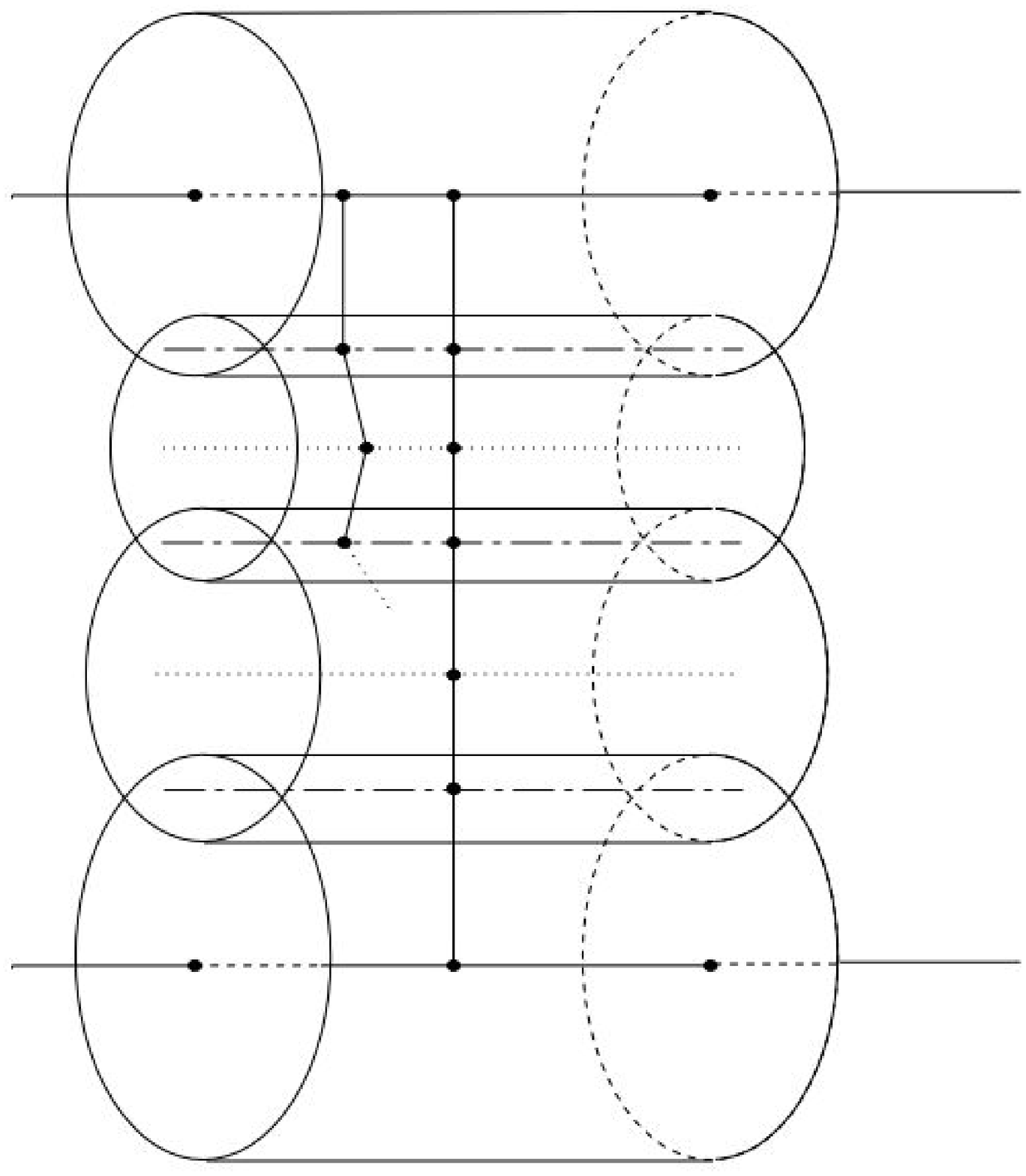}
\put(40,79){$\gamma(0)$}\put(40,67){$\gamma(s_1)$}\put(40,58){$\gamma(t_1)$}\put(40,51){$\gamma(s_2)$}\put(40,40){$\gamma(t_2)$}\put(40,30){$\gamma(s_3)$}\put(40,15){$\gamma(1)$}\put(78,79){$N_{0}$}\put(78,15){$N_{1}$}\put(30,80){$x$}\put(30,75){$c_1$}\put(31,65){$c_2$}
\end{overpic} 
\end{center}
\vspace{5mm}
\end{minipage}
\begin{minipage}{0.51\textwidth}
\vspace{-4mm}Around the points $\gamma(0)$ and $\gamma(1),$ let the plaques $P_0 \subset N_0$ and $P_1 \subset N_1$  be centers of tubular neighborhoods. The centers of all neighborhoods of the covering shall intersect $\gamma.$ Furthermore, choose a minimal subdivision $$0 = t_0 < t_1 < \ldots < t_k < t_{k+1} = 1,$$ where $P_{t_i}$ is such a centering plaque for all $i,$ i.e. each point on $\gamma$ lies either in one or in the intersection of two consecutive tubular neighborhoods. Let $\{s_i\}_{i=1}^{k+1} \subset [0, 1]$ be chosen in such a way that $$t_0 < s_1 < t_1 < s_2 < \ldots < t_k < s_{k+1} < t_{k+1}$$ and $\gamma(s_i)$ lies in the intersection: $$\gamma(s_i)\in {\rm Tub}(P_{t_{i-1}}) \cap {\rm Tub}(P_{t_i}).$$
\end{minipage}
Now, let $U \subset P_0$ be a neighborhood -- sufficiently small -- of $\gamma(0)$ and $x \in U.$ 
As $\gamma$ is contained completely in the finite, open covering of tubular neighborhoods, there exists a $U$ so small that the following construction of a broken geodesic is possible within the covering.

In the neighborhood ${\rm Tub}(P_0),$ there is a minimal geodesic $c_{1} : [0,1] \to {\rm Tub}(P_0),$ with $c_{1}(0) = x$ and $c_{1}(1) \in N_{s_{1}},$ and the length of $c_{1}$ is equal to the length of $\gamma$ restricted to $[0, s_{1}].$ Similarly, there is a minimal geodesic $c_2$ from $c_1(1) \subset N_{s_{1}}$ to $P_{t_1}$ of length $L(\gamma|_{[s_1, t_{1}]}).$ Proceeding like that,  we obtain a broken geodesic $c = c_1 c_2 \ldots c_{2k+2}$ of length $L(c) = L(\gamma) = d(\gamma(0), P_1)$ which starts in $x \in U$ and ends in $P_1.$ Thus, the distance between $x$ and $P_1$ has to be less equal than or equal to $L(c) = L(\gamma).$

But $\gamma$ is locally minimizing and equality follows by $$L(\gamma) \le d(x, P_1) \le L(\gamma).$$

To prove step (ii), let $N \in \F$ be arbitrary. Furthermore, choose $\eps > 0$ in such a way that there is a plaque $P_N \subset N$ with tubular neighborhood ${\rm Tub}_{\varepsilon}(P_N).$ In addition, we can assume that the leaf $N$ meets the set ${\rm exp}_N^\perp(B_\eps(P_N))$ just once, namely as centering leaf.

Now, let us assume that $\inf\{d(p, \mathcal C_N) \mid p \in N\} = 0.$ As the leaves are properly embedded submanifolds, Proposition \ref{propcutlocus} tells us that the cut locus function $\sigma_N$ is continuous. Thus, there exists a leaf $L$ with $d(P_L, P_N) = \frac{\eps}{3}$ where $P_L$ is a connected component of $L \cap {\rm Tub}_{\varepsilon}(P_N),$ there exists a point $q \in N$ with  $d(q, \mathcal C_N) < \frac{\eps}{3},$ and by (i), there exists a horizontal geodesic $\gamma$ that realizes the distance $d(q, \mathcal C_N)$ and meets the leaf $L$ at length $\frac{\eps}{3}.$
But $\gamma$ cannot be minimal between $q \in N$ and $L$ because the cut locus of $N$ meets $\gamma$ previous to $L.$ Thus, $\delta := d(q, L) < \frac{\eps}{3}$. Sliding along the leaves with this local distance from $L$ to $N$ up to $P_N$ which is possible by (i) shows that there is a plaque of $L$ lying within the tubular neighborhood ${\rm Tub}_{\varepsilon}(P_N)$ on the $\delta$-sphere of $N.$ But $P_L \subset L$ lies on the $\frac{\eps}{3}$-sphere around $P_N.$ Once again, sliding along the leaves with distance $\delta$ up to $P_L$ shows that there must be one more plaque of $N$ within the neighborhood ${\rm Tub}_{\varepsilon}(P_N),$ a contradiction to the choice of $\eps$!

\begin{minipage}[t]{1\textwidth}
\vspace{1mm}
\hspace{-5mm}\begin{overpic}[width=12cm,
]{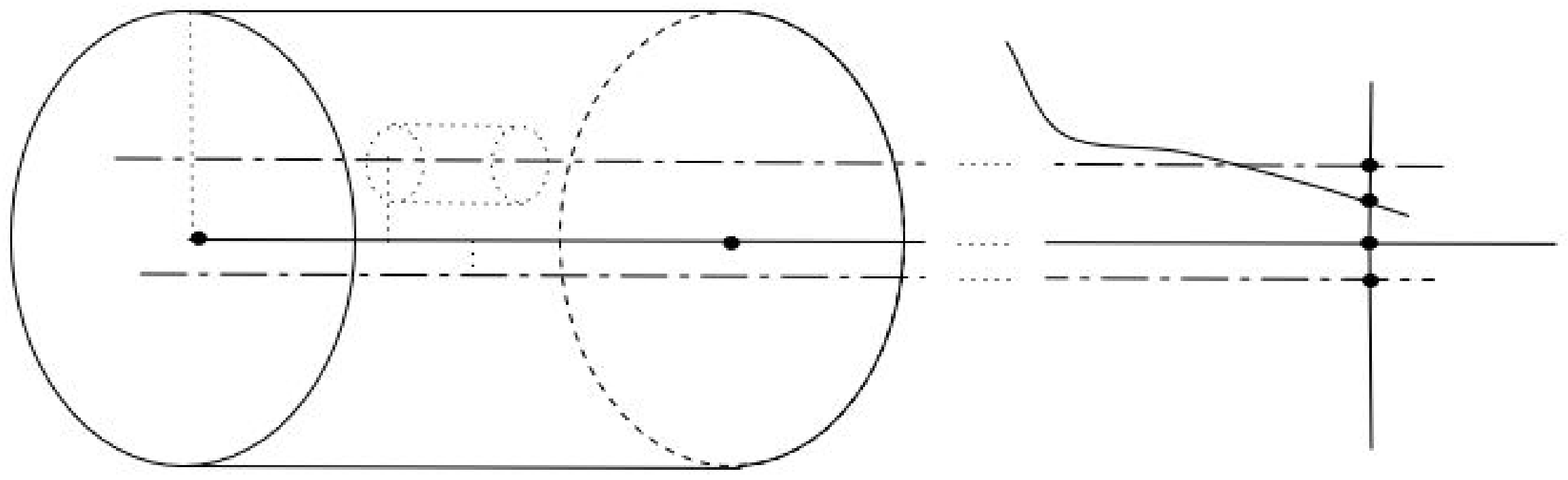}
\put(63,34){$\mathcal C_N$}\put(87,19){$q$}\put(87,8){$\gamma$}\put(7,36){${\rm Tub}_\eps(P_N)$}\put(92,24){$L$}\put(96,18){$N$}\put(67,16){$L$}\put(23.5,22){$\frac{\eps}{3}$}\put(31,18.5){$\delta$}\put(24,28.5){${\rm Tub}_\delta(P_L)$}
\end{overpic} 
\vspace{5mm}
\end{minipage}
So, $\inf\{d(p, \mathcal C_N) \mid p \in N\} > 0$ has to hold.
\end{proof}


\subsection{Holonomy of leaves}

In \cite{molino}, the notion of \mbox{\emph{leaf holonomy}} for Riemannian foliations is defined by sliding along the leaves (see Subsection \ref{ss:sliding} here) on the regular stratum $\Sigma_{\rm reg}:$

If $c$ is a loop in $x$ and if we choose the local transversals $T = T',$ then $h_c$ is the germ of a local diffeomorphism with fixed point $x.$ The composition of two loops $c_1, c_2$ in $x$ is  $h_{c_1c_2} = h_{c_1}\circ h_{c_2}.$ Thus, we obtain a group homomorphism $$h_x : \pi_1(L, x) \to {\rm Diff}_x(T),$$
whose image is the \emph{holonomy group of $L$ in $x$.} By identification of a local transversal with the quotient $\bar U,$ the holonomy group is independent of $T.$

A regular leaf is called \emph{exceptional}, if its holonomy is non-trivial. Furthermore, we call a regular leaf $N$ \emph{specificly regular}, if there is a $p$ in $N$ and a local tubular neighborhood around $p$ wherein every meeting leaf has exactly one plaque. Note that only under the condition of being properly embedded, leaves with trivial holonomy are just the same as specificly regular leaves as the following lemma will show. Thus, we denote this stratum of leaves by $\Sigma_{\rm triv}.$
 
\begin{lemma}\label{lemhol}
Let $(M, \mathcal F, g)$ be singular Riemannian with properly embedded leaves. Then, a regular leaf $N$ has non-trivial holonomy in $x$ if and only if for every global tubular neighborhood of $N$ there exists a leaf meeting the associated slice $S(x)$ more than once, i.e. if and only if $N$ is notspecificly regular.
\end{lemma}

To become acquainted with exceptional and specificly regular leaves, we state some easy facts about them previous to the proof of the lemma.
 
\begin{remark}
\begin{trivlist}
\item{(i)} For a leaf $L$, the number of components of the intersection $L \cap S(x), x \in N,$ is independent of $x \in N.$ \setlength\labelwidth{-0,3cm}\item{(ii)} The holonomy of a leaf is non-trivial in every point of the leaf if it is non-trivial in an arbitrary point.
\item{(iii)} For specificly regular leaves, the number of components of the intersection as stated in (i) is constant along connected components of $\Sigma_{\rm triv} \cap S(x).$
\item{(iv)} The plaques of a leaf $B$ lying in a sufficiently small, global tubular neighborhood of $N$ are lying in \emph{one} tube around $N$ because of globally constant distance.
\end{trivlist}
\end{remark}

\begin{proof} The first item results from Theorem \ref{theotube} and locally constant distances of leaves, the second is a direct conclusion of the first. To see the third item, we consider a tubular neighborhood of a specificly regular leaf $B$ so small that it is contained completely in the global tubular neighborhood of $N.$ Now, combine the local description of a Riemannian foliation with sliding along $B$ to see that the third property is open on $\Sigma_{\rm triv} \cap S(x).$ Closedness follows since the complement is open.  ``Sufficiently small'' in the fourth item is easily seen to be less than one half of the radius of the global tubular neighborhood (cf. argument (ii) in the proof of theorem \ref{theotube}).
\end{proof}

\noindent \emph{Proof of Lemma \ref{lemhol}.}
``$\Rightarrow$''\ Assume that there is a tubular neighborhood with slice $S(x)$ of $x$ such that every meeting leaf has exactly one plaque in it. As the leaves have the same dimension as $N,$ the intersection of a leaf with the slice is one single point, and sliding along the leaves just produces the identity map.

``$\Leftarrow$''\ Consider a leaf $N$ with global tubular neighborhood of radius, say, $\eps;$ and let $x$ be a point of $N.$ Further, let $B$ be a leaf with at least two intersection points with the slice $S_{\eps}(x),$ denote these points by $b_1$ and $b_2.$ Sliding along the leaf $N$ of the distance  $d(b_1, x)$ to $b_2$ therefore defines a non-trivial element in the holonomy group of $x$ in $N.$ As $B$ was arbitrary, the claim follows. 
\qed\\

Now, we will show that the exceptional leaves are not as bad as they may seem at a first glance: the stratum $\Sigma_{\rm triv}$ of leaves of trivial holonomy (i.e. specificly regular leaves) is open, connected, and dense in $M;$ beyond, the stratum of exceptional leaves $\Sigma_{\rm ex} = \Sigma_{\rm reg} \setminus \Sigma_{\rm triv}$ owns -- at least locally -- a submanifold structure.

\begin{proposition}\label{propregdense} Let $(M, \mathcal F)$ be a singular Riemannian foliation with properly embedded leaves. Then, the set of specificly regular leaves is open and dense in $M$.
\end{proposition}

\begin{proof}
The stratum of regular leaves $\Sigma_{\rm reg}$ is open and dense in $M.$ Therefore, it is sufficient to show that $\Sigma_{\rm triv}$ is open and dense in $\Sigma_{\rm reg}.$ By Lemma \ref{lemhol}, the first claim follows.

Let $N$ be exceptional, $x \in N$ and $\eps_0 > 0$ such that ${\rm Tub}_{\eps_0}(N)$ is a global tubular neighborhood of $N.$ Let $\eps < \eps_0$ be arbitrary. If every leaf intersected the $\frac{\eps}{3}$-slice of $x \in N$ only once, $N$ would not be exceptional, so we can find a leaf $B$ intersecting $b$ times. Assume that $B$ is not specificly regular, then iterate the previous consideration as follows:

All the points $\{B_1, \ldots, B_b\} = B \cap S_{\frac{\eps}{3}}(x)$ lie in a normal distance sphere of $N$ around $x.$ Around each of these points choose local tubular neighborhoods $\{{\rm Tub}_\delta(B_i)\}$ with fixed radius $\delta$ in such a way that they do not intersect, they lie within the $\frac{\eps}{3}$-neighborhood of $N,$ and that each contains just one $B_i.$ Again, have a look at the neighborhoods of one third of this radius. As $B$ is exceptional, there is a leaf $C$ meeting each neighborhood several times and therefore meeting ${\rm Tub}_{\frac{\delta}{3}}(B_i) \cap S_{\frac{\eps}{3}}(x)$ more than once, say $c$ times. Thus, there are $bc$ different points of $C$ in the $\frac{\eps}{3}$-slice of $x,$ and
\begin{trivlist}
\item{(i)} the points $C_{i1}, \ldots, C_{ic}$ are at the same distance from $B_i,$ 
\item{(ii)} all the $bc$ points $C_{11}, \ldots, C_{bc}$ lie on a distance sphere of fixed radius around $x.$
\end{trivlist} Let $n$ be the dimension of $M$ and $k$ the dimension of the regular leaves. Then, the dimension of a normal distance sphere of a point in $\Sigma_{\rm reg}$ is equal to $(n-k)-1.$  The different points of $C \cap S_{\eps}(x)$ live therefore in an -- at most -- $(n-k-2)$-dimensional set. Thus, the iteration argument has to end after $n-k-1$ steps which means finding a specificly regular leaf.

The leaf $N,$ the point $x,$ and the radius $\eps$ were chosen arbitrarly, so the claim follows.
\end{proof}

To prove that $\Sigma_{\rm triv}$ is connected, we will use the structure of the exceptional stratum, which can be stated as follows:

\begin{lemma}\label{lemexmfd}
Each connected component of $\Sigma_{\rm ex}$ is a submanifold of $\Sigma_{\rm reg}$ with smooth (relative) interior, namely either being an isolated exceptional leaf or locally being the pre-image of totally geodesic, convex sets of the (local) quotient space.
\end{lemma}

\begin{remark} The following proof does not give any information about smoothness of the boundaries of $\Sigma_{\rm ex}$ in $\Sigma_{\rm reg}$ resp.~in $M.$ It will follow by Lemma \ref{lemacc} that every component actually has no boundary (this fact is not used in the meantime).
\end{remark}

\begin{proof} Assume that there is a leaf $N$, to which exceptional leaves accumulate. It is a direct consequence that $N$ itself is exceptional or singular. Let ${\rm Tub}_\eps(N)$ be a global tubular neighborhood and let $A$ be exceptional, lying within this neighborhood. Then there exists a horizontal geodesic realising the distance between $A$ and $N.$ By globally constant distance of leaves and by connectedness, we conclude that there is a differentiable strip between the two leaves consisting of minimal geodesics $\{\gamma_a\}_{a \in A}$ of length $l$. If $B$ is another leaf meeting one of these geodesics at time $t,$ then $B$ meets every such geodesic at time $t$. This follows by an open-closed argument. To see that $B$ lies totally within this strip, we argue as follows:

Take the distance from $p \in B$ to $N$ as well as the distance from $p$ to $A$ and look at the associated broken geodesic from $\pi_N(p) \in N$ to $\pi_A(p) \in A$. Piecewise, it is of the same length as the geodesics of the strip, so either it must be unbroken (and thus contained in $\{\gamma_a\}_{a \in A}$) or the distance from $\pi_N(p) \in N$ to $\pi_A(p) \in A$ must be less than $l$, which is impossible.

Lemma \ref{lemposhom} tells us that $B$ has the same dimension as $A$. Let ${\rm Tub}_\delta(A)$ be a global tubular neighborhood of $A$ and specify $B$ to be a leaf lying within this neighborhood. Choose $a \in A$ arbitrarily and $b \in B$ lying on the unique minimal geodesic from $a$ to $N$ (see the figure below). Within a slice through $a,$ the point $b$ is the only one belonging to $B$ because the distance spheres $S_{d(a,b)}(a)$ and $S_{d(b, N)}(\gamma_a(l))$ intersect only in $b.$ Because $A$ is exceptional, for every neighborhood $U(b) \subset {\rm Tub}_\delta(A)$ of $b$ there is a specificly regular leaf which intersects the slice $S_\delta(a)$ more than once.

\begin{minipage}[t]{1\textwidth}
\vspace{1mm}
\begin{center}
\begin{overpic}[width=6cm,
]{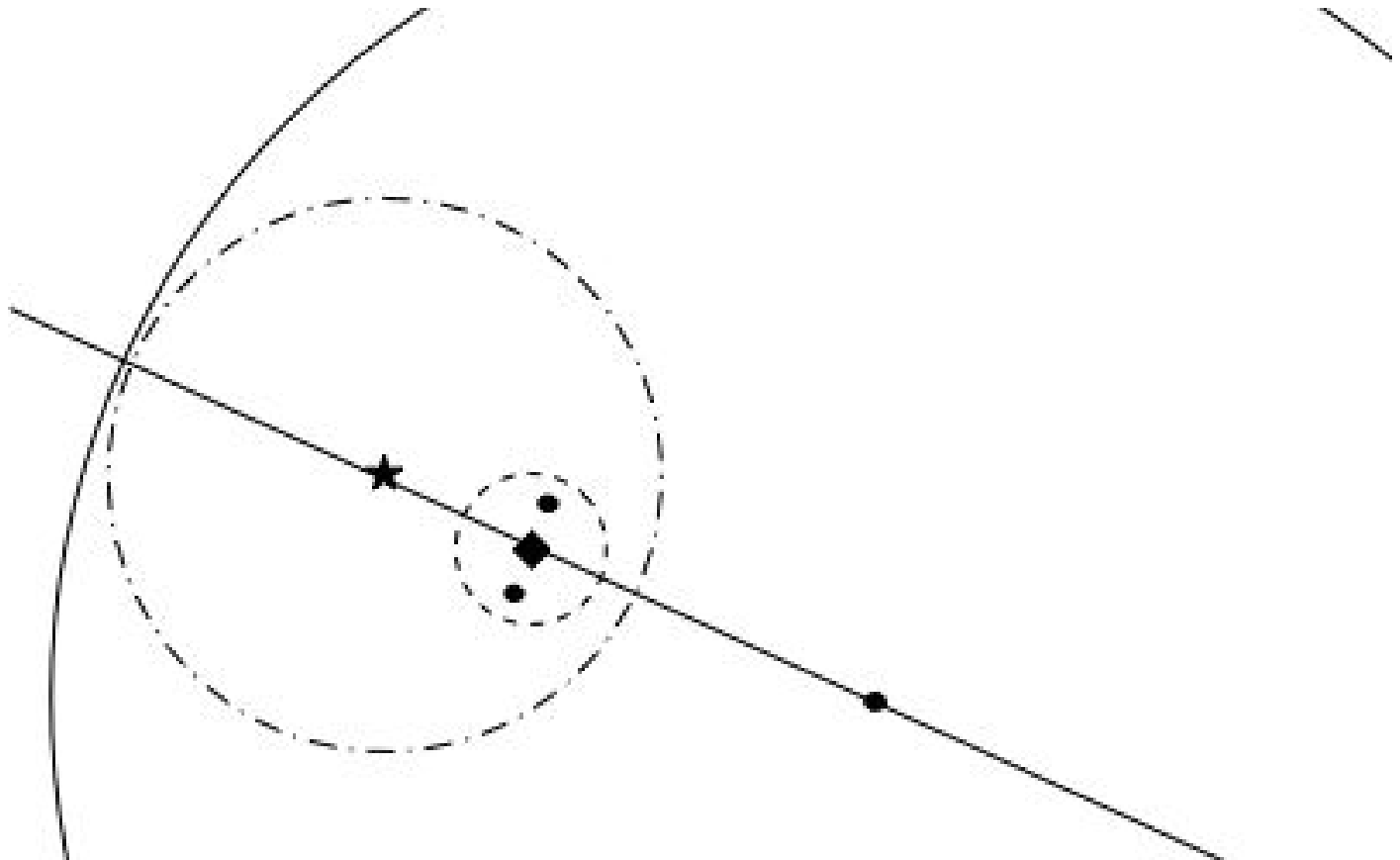}
\put(24,23){$a$}\put(39,17){$b$}\put(66,12){$\gamma_a(l)$}\put(60,50){${S_\eps(\gamma_a(l))\subset\rm Tub}_\eps(N)$}\put(21,39){$S_\delta(A)$}\put(45,22){$U(b)$}
\end{overpic} 
\end{center}
\vspace{3mm}
\end{minipage}
But locally constant distance forces this leaf to meet $U(b) \cap S_\delta(a)$ several times, elsewise there would be more points of $B$ in $S_\delta(a).$ By definition, $B$ must be exceptional, too.

This consideration holds for every leaf $B \subset \bigcup_{a \in A} \gamma_a \cap {\rm Tub}_\delta(A).$ Within $\Sigma_{\rm reg},$ the stratum $\Sigma_{\rm ex}$ is closed. Thus, a boundary leaf $B \in \bigcup_{a \in A} \gamma_a \cap \overline{{\rm Tub}_\delta(A)}$ also is exceptional. Assigning the previous part of $A$ to this boundary leaf $B$ and repeating the proof yields that all leaves of the strip $\bigcup_{a \in A} \gamma_a$ between  $N$ and $A$ are exceptional.

In the local quotient space $\bar U,$ minimal geodesics connecting arbitrary points of  $\pi_U(\Sigma_{\rm ex} \cap U)$ thus lie completely in $\pi_U(\Sigma_{\rm ex} \cap U).$ This set is locally convex. Therefore, its interior is well-defined and can be endowed with a smooth submanifold structure, see p.~139 of \cite{cheeger}. Hence, the preimage $\Sigma_{\rm ex} \cap U$ of the submersion $\pi_U$ is smooth in its interior, too.
\end{proof}

\begin{proposition}\label{propregcon}
The set $\Sigma_{\rm triv}$ of leaves with trivial holonomy is connected.
\end{proposition}

\begin{proof} Assume $\Sigma_{\rm reg} \setminus \Sigma_{\rm ex}$ to be not connected. The component $\mathcal Z \subset \Sigma_{\rm ex}$ that separates the stratum $\Sigma_{\rm triv}$ has to be $(n-1)$-dimensional and without boundary. Hence, the fibers of the normal bundle of $\mathcal Z$ are one-dimensional.

First case: An exceptional leaf $N$ and thus all regular leaves are of dimension $n-1$. Within every local tubular neighborhood (sufficiently small) of  $p \in N$, there is -- by lemma \ref{lemhol} -- a regular leaf having several plaques within a corresponding tube. If $N$ separates the specificly regular stratum, $N$ separates the tube. Transversally, meaning within the quotient, the tube consists of only two points, which then have to belong to different components of $\Sigma_{\rm triv}.$ But both belong to the same leaf which is connected, a contradiction.

Second case: The component $\mathcal Z$ is a union of exceptional leaves. Let $p \in L_p \subset \mathcal Z$ and ${\rm Tub}_\eps(L_p)$ be a global tubular neighborhood of $L_p.$ By assumption, ${\rm Tub}_\eps(L_p) \cap \Sigma_{\rm triv}$ is separated by $\mathcal Z.$ Beyond, there is a leaf $N \subset \Sigma_{\rm triv}$ arbitrarily close to $p$ which intersects $S_\eps(p)$ several times. Choose $x,y \in N \cap S_\eps(p),$ and a path $c: [0,1] \to N$ from $x = c(0)$ to $y = c(1).$ At every point $c(t), t \in [0,1],$ take the distance $d(c(t), \mathcal Z).$ This is constant along $c$, hence $d(x, \mathcal Z) = d(y, \mathcal Z).$ As $\mathcal Z$ is a $(n-1)$-dimensional submanifold of $\Sigma_{\rm reg},$ the points $x$ and $y$ have to be in different components of ${\rm Tub}_\eps(L_p) \cap \Sigma_{\rm triv}.$ But again, both belong to the same leaf, a contradiction.
\end{proof}

By remark (iii) coupled with Propositions \ref{propregdense} and \ref{propregcon} we conclude:

\begin{corollary}\label{cortype}
In a global tubular neighborhood of an exceptional leaf, all specificly regular leaves intersect a slice equally often (and more than once).\qed
\end{corollary}

\subsection{The 
stratification of $\Sigma_{\rm ex}$ and $M$} The last subsection, particularly the last Corollary \ref{cortype}, leads to the following, well-defined notion:

\begin{definition}\label{deftype}
The \emph{type of an exceptional leaf $N_x$} is the natural number $$a_N := \# \{S_\eps(x) \cap L \mid L \subset \Sigma_{\rm triv} \cap  {\rm Tub}_\eps(N)\}.$$
\end{definition}

Combining the local description of a singular Riemannian foliation with the Tube Theorem \ref{theotube} and Corollary \ref{cortype} yields

\begin{lemma}\label{lemcovering}
Let $N$ be exceptional of type $a_N$ resp.~ specificly regular with global tubular neighborhood ${\rm Tub}_\eps(N).$ Then, the orthogonal projection $$\pi: {\rm Tub}_\eps(N) \to N$$ restricted to a leaf $L \subset {\rm Tub}_\eps(N)$ is an $\frac{a_N}{a_L}$-fold covering resp.~ a global diffeomorphism.\hfill \qed
\end{lemma}

A first step towards the stratification shall be the following lemma generalising the claim of Lemma \ref{lemexmfd} as announced. Let $c: (-\eps, \eps) \to M$ be a horizontal geodesic.

\begin{lemma}\label{lemacc}
If $c(0)$ is an accumulation point of both exceptional points and specificly regular points  along $c$, so $c(0)$ must be singular.
\end{lemma}   

\begin{remark}
This lemma holds for singular Riemannian foliations with properly embedded leaves. Within the restriction to foliations without horizontally conjugate points, we have the stronger Proposition \ref{propiso}.
\end{remark}

\begin{proof}
Without restriction and by Lemma \ref{lemexmfd} we can assume $\eps > 0$ to be so small that ${\rm Tub}_\eps (N_0)$ is globally defined and $c|_{(-\eps, 0)}$ only consists of exceptional points whereas on $c|_{(0, \eps)},$ there only lie specificly regular points. Furthermore, assume $N_0$ to be exceptional (by Proposition \ref{propregdense}, it is clear that $N_0$ cannot be specificly regular).

Let $x \in \Sigma_{\rm triv} \cap c$ such that $d(x, c(0)) = \frac{\eps}{4}.$ The corresponding leaf $N_x$ meets the tube ${\rm Tub}_\eps(N_0)$ more than once, namely $a_{N_0}$-times (see the following figure).

\begin{minipage}[t]{1\textwidth}
\vspace{1mm}
\begin{center}
\begin{overpic}[width=7cm,
]{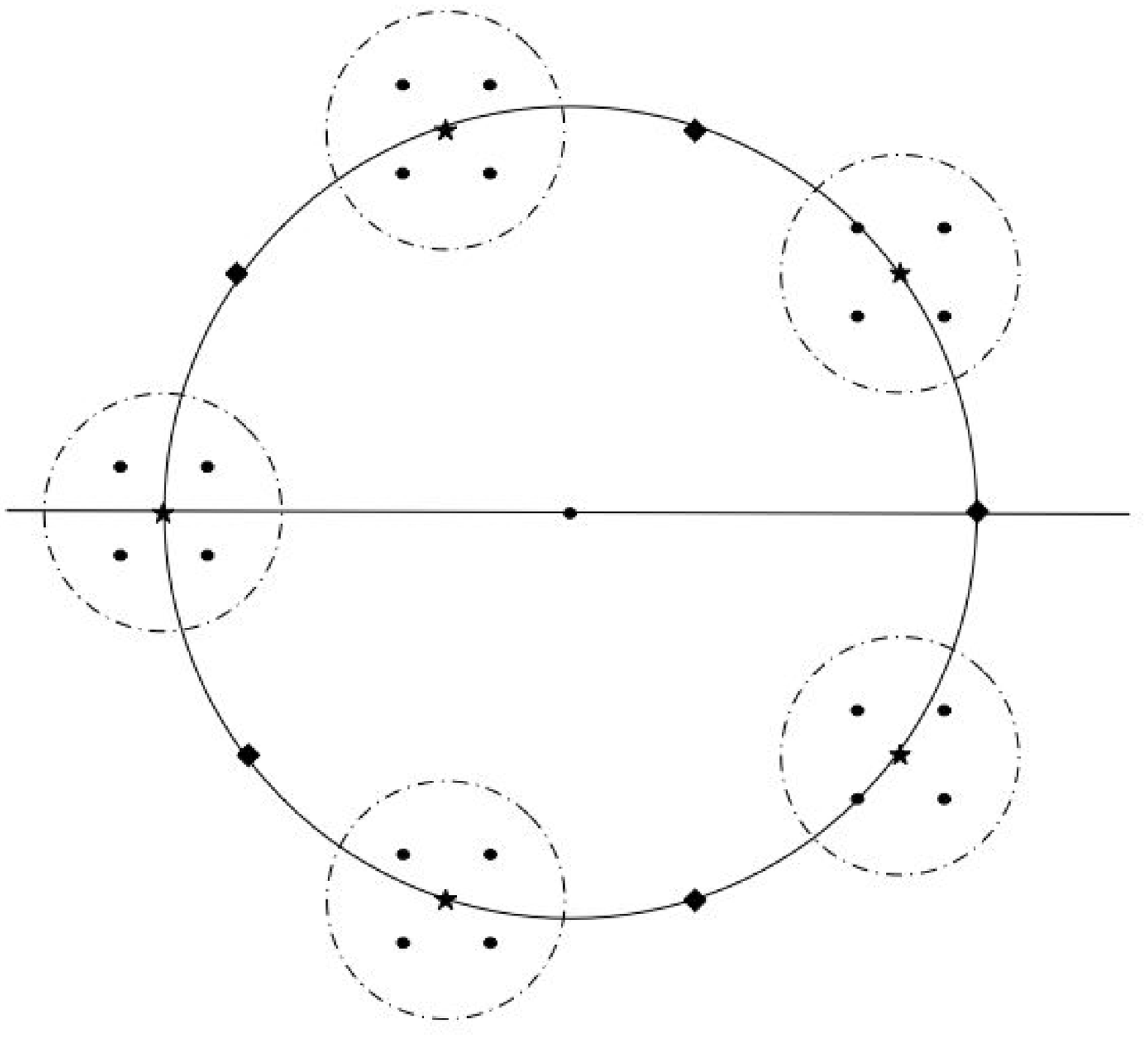}
\put(17,43){$y$}\put(86,43){$x$}\put(50,42){$c(0)$}\put(94,43){$c$}\put(35,55){$S_{\frac{\eps}{4}}(c(0))$}
\put(61,79.5){$x_i$}\put(38,9){$y_i$}\put(-2,55){${\rm Tub}(N_y)$}
\end{overpic} 
\end{center}
\vspace{2mm}
\end{minipage}

All intersection points $\{x_i\}$ with the slice through $c(0)$ lie on the distance sphere $S_{\frac{\eps}{4}}^{c(0)}.$ We know that $y := c(-\frac{\eps}{4}) \in \Sigma_{\rm ex}$ and of course $d(x,y) = \frac{\eps}{2}.$ Sliding of this distance along the leaf $N_x$ to the different points $\{x_i\}$ yields the fact that the leaf $N_y,$ too, meets the tube ${\rm Tub}_\eps(N_0)$ several times. Because of the fixed distance to $N_0,$ the intersection points $\{y_i\}$ of $N_y$ and the slice of $c(0)$ lie on the sphere $S_{\frac{\eps}{4}}^{c(0)}.$ By $d(x,y) = \frac{\eps}{2} = d(x_i, y_i)$ for all $i,$ wo conclude, that each $y_i$ is the antipode of the corresponding $x_i.$ But $N_y$ is exceptional, thus we find specificly regular leaves with several plaques within local tubular neighborhoods around the $y_i.$ Choose such neighboorhoods so small that neither $c(0)$ nor one of the $x_j$ is contained in it. Hence, a specificly regular leaf that defines the type $a_{N_y}$ must have at least $a_{N_y}a_{N_0}$ plaques within ${\rm Tub}_\eps(N_0)$ around $N_0.$ This contradiction shows that $N_0$ is singular. 
\end{proof}

From the beginning, we have called the set $\Sigma_{\rm ex}$ a stratum; now we will show that this intuitional naming is correct in the sense of the following common definition. 

A locally finite partition of a manifold $M$ into locally closed, connected submanifolds $M_i \ (i \in I)$ of $M$, the \emph{strata,} is called a \emph{stratification} if the following condition holds:
For every $i \in I$ denote the closure of $M_i \subset M$ by $\overline M_i.$ Then
$$\overline M_i = M_i \cup \bigcup_{j \in I_i} M_j, \text{ wobei } I_i \subset I \setminus \{i\} \text{ and } \dim M_j < \dim M_i \text{ for all } j \in I_i.$$

\begin{remark}
Recall that in the case of a singular Riemannian foliation coming from the orbit decomposition of a proper group action, the subdivision into orbit types is not only a stratification but also fulfills the tougher technical conditions of a Whitney stratification. We conjecture that this is the case in general.
\end{remark}

The last argument to do for the stratification by types of leaves as in \ref{deftype} is

\begin{lemma}\label{lemacctype}
If $p$ is an accumulation point of exceptional points $\{p_i\}_{i= 1}^\infty$ and not singular, then $N_{p}$ is exceptional of type
$$a_{N_{p}} \ge \limsup \{a_{N_i}\}, \text{ where }a_{N_i} := a_{N_{p_i}}.$$
Beyond, the function $x \mapsto a_{N_x}$ defined along a horizontal geodesic $c \subset \Sigma_{\rm ex}$ is constant except for isolated points, i.e. nearly all leaves intersecting $c$ are of the same type, the remaining points are of higher type.
\end{lemma}

\begin{proof}
For every tubular neighborhood of $p$ there is an $i_0 \in \N$ such that the exceptional leaves  $\{N_{p_i}\}_{i \ge i_0}$ meet that neighborhood. Additionally, for every $i \ge i_0$ there are specificly regular leaves, meeting the neighborhood $a_{N_i}$-times, namely those who define the type of $N_{p_i}.$ By assumption, the leaves are properly embedded, therefore the number $a_{N_{p}}$ and thus the right hand side of the inequality are finite, and the first claim follows.

Now, restrict the attention to a connected subset $U \subset \Sigma_{\rm ex}$ and define the set of leaves of minimal type, which is open in $U.$ Using the methods of the proof of Proposition \ref{propregdense}, as there are taking thirds of tubular neighborhoods, sliding of distances, and iteration, we can interpret these leaves as \emph{regular in $U$,} i.e. they have to meet tubular neighborhoods of exceptional leaves of higher type more than once. Hence, if two different types of leaves accumulate to one point along a horizontal geodesic in $\Sigma_{\rm ex},$ we can carry over the Lemma \ref{lemacc} directly and conclude that the accumulation point has to be singular, a contradiction.
\end{proof}

With the work already done, the technical Lemma \ref{lemacctype} can be reformulated as
\begin{proposition}\label{proptype}
The decomposition of $\Sigma_{\rm ex}$ with respect to the types of the leaves
$$\Sigma_{\rm ex} = \dot\bigcup_{a_N} \{N \text{ is of type } a_N\}$$ is a stratification.\hfill\qed
\end{proposition}

Molino has shown that the decomposition of $(M;\F)$ with respect to the dimension of the leaves is a stratification. Thus, we conclude with the help of Lemma \ref{lemacctype} and Proposition \ref{proptype}:

\begin{theorem}\label{theostratification}
The finer decomposition of a singular Riemannian foliation $(M,\F)$ with properly embedded leaves into the sets of specificly regular leaves, strata of types of exceptional leaves, and dimensional strata of singular leaves is a stratification.\hfill\qed
\end{theorem}

 \section{Foliations without horizontally conjugate points}\label{s:whcp}

The notion of a foliation without horizontally conjugate points was introduced by Lytchak and Thorbergsson  in \cite{lth} as follows. 

\begin{definition}\label{defwhcp}
A singular Riemannian foliation is \emph{without horizontally conjugate points} if the following holds for all leaves $N$ and all geodesics $\gamma$ perpendicular to $N$: Every $N$-Jacobi field along $\gamma$ that is tangential to a leaf different from $N$ is tangential to every leaf met by $\gamma.$
\end{definition}

We call a point $q \in M$  \emph{horizontally conjugate to $p$} if there is a horizontal geodesic $\gamma: [0,1] \to M$ with $\gamma(0) = p, \gamma(1) = q$ and an $N_p$-Jacobi field $J$ along $\gamma$ such that $J(0) \in T_pN_p, J(1) \in T_qN_q,$ but $J$ is not everywhere tangential.
 
\subsection{Tangential Jacobi fields, homothetic transformations}
 
Let $(M, \F)$ be a singular Riemannian foliation and let $\gamma$ be a horizontal geodesic. As in \cite{lth} we define the vector space
$$\mathcal K := \bigcap_t \{J \mid J \text{ is a }N_t \text{-Jacobi field along }\gamma\},$$ wherby the $t$'s cover the domain of definition of $\gamma.$

We will call a Jacobi field $J$ along $\gamma$ \emph{everywhere tangential,} if for all $t$ it is tangential to the leaf $N_t.$ 
Thus, every $J \in \mathcal K$ is everywhere tangential.

At least if the foliation has no horizontally conjugate points, we can reverse this implication as follows.

\begin{lemma}\label{lemjacobifield}
Let $(M, \F)$ be a singular Riemannian foliation  without horizontally conjugate points. If $J$ is an $N_{t_0}$-Jacobi field that is everywhere tangential along $\gamma,$ then $J$ is an $N_{t}$-Jacobi field for all $t \in \R$, i.e. $J \in \mathcal K$.
\end{lemma}

\begin{remark}\label{Prop3.1}
This lemma is a non-mentioned detail of the proof of Proposition 3.1 in \cite{lth} which states that,  within the notation used above, $$T_{\gamma(t)}N_t = \{J(t) \mid J \in \mathcal K\}.$$

In that article is shown that for every $t_0$ in the domain of $\gamma$ and for every $v \in T_{\gamma(t_0)}N_{t_0}$ there is a $N_{t_0}$-Jacobi field $J$ which is tangential to the corresponding leaf through $\gamma(t)$ for all $t$. To prove this, a curve $\alpha \subset N_{t_0}$ with $\dot\alpha(0) = v$ is lifted to a neighboring leaf $B$ within a tubular neighborhood of $\gamma(t_0)$ via the local Riemannian submersion $\pi_{N_{t_0}}.$ The geodesics pointwise connecting both curves define a horizontal variation $f: (-\varepsilon, \varepsilon) \times \R \rightarrow M$ of $\gamma.$ The appropriate Jacobi field $J$ along $\gamma$ is tangential twice, namely in $N_{t_0}$ and in $B,$ hence -- by assumption -- everywhere.

As $\alpha$ lies in $N_{t_0},$ there actually is a family of Jacobi fields $J_s$ (along the geodesics $f_s(\cdot )$), that are tangential to the leaves $N_{f_s(t)}$ for all $t.$ 
\end{remark}

\noindent\emph{Proof of Lemma \ref{lemjacobifield}.}
Fix $t$ and define $c : (-\varepsilon, \varepsilon) \rightarrow M,\ c(s) := f(s, t),$ to be a path through the variation $f$ at time ${t}$. Then, $$\dot c(0) \in T_{\gamma (t)}N_{t} \quad \text{and} \quad \dot c(s) \in T_{f_s(t)}N_{f_s(t)}.$$

We now verify that -- if necessary diminuish $\varepsilon$ -- for all $s \in (-\varepsilon, \varepsilon),$ the geodesics $f_{s}(\cdot)$ meet the leaf $N_{t}$ at time ${t}.$ For this purpose let $s_0 \in (-\varepsilon, \varepsilon)$ be a point with $$\dim(N_{c(s_0)}) = \max \{\dim(N_{c(s)}) \mid s \in (-\varepsilon, \varepsilon)\}.$$ Without loss of generality consider $s_0$ to be in the interval $(-\varepsilon,0)$. There is a neighborhood of $s_0$ in $M$ containing leaves of the same or higher dimension only. By maximality of the leaf dimension in $c(s_0)$ along $c,$ there is a $\delta >0 $ with $c |_{(s_0 - \delta,\, s_0 + \delta)} \subset \Sigma_{\dim(N_{c(s_0)})}$. Restrict the attention to the stratum $\Sigma_{\dim(N_{c(s_0)})}$ to see that an everwhere tangential path cannot cross several leaves:

For $c(s_0),$ there is a local, centered chart $\psi: U_{s_0} \rightarrow V \times W \subset \R^k \times \R^{n-k}$ such that $\psi(U_{s_0}\cap N_{s_0}) \subset \R^k,$ and if $N$ is a leaf with $N \cap U_{s_0} \neq \emptyset,$ then $$\psi(N \cap U_{s_0}) \subset \R^k \times \{\text{isolated points}\}.$$ Denote the composition $\psi \circ c : (s_0 - \delta,\, s_0 + \delta) \rightarrow V \times W $ by two components $(c_1, c_2)$.  As $\dot c(s) \in T_{f_s(s_0)}N_{f_s(s_0)}$, we conclude $\dot c_2 \equiv 0$ and $c(s) \in N_{s_0}$ for all $s \in (s_0 - \delta,\, s_0 + \delta).$ 

Let $\tilde \delta$ be maximal in the sense that $c |_{(s_0 - \delta,\, s_0 + \tilde\delta)}$ is contained in the leaf $ N_{c(s_0)}.$ This interval then is open as seen above. We now look at the boundary point $c(s_0 + \tilde \delta):$ Consider a tubular neighborhood of the leaf $N_{c(s_0 + \tilde \delta)}$, wherein the plaques of the leaves have locally constant distance to $N_{c(s_0 + \tilde \delta)}$. The path $c |_{[s_0 ,\, s_0 + \tilde\delta]}$ is compact, hence $c$ crosses finitely many plaques during this time. Thus, there is a $\rho > 0$ such that the distance from $c(s_0 + \tilde \delta)$ to $c(s)$ is bigger than or equal to $\rho$ for all $s \in [s_0 ,\, s_0 + \tilde\delta),$ contradicting continuity of $c$. Hence, $c |_{(-\varepsilon, \varepsilon)} \subset N_{c(s_0)} = N_{c(0)} = N_{t}.$ 

As $f_s(\cdot )$ is perpendicular to $N_{t_0}$ for all $s,$ $f_s(\cdot )$ is perpendicular to $N_{t}.$ Therefore, $f$ is a $N_{t}\;$-variation, and $J$ is a $N_{t}\;$-Jacobi field. 
\qed\\

The preceding considerations reveal for singular Riemannian foliations

\begin{proposition}\label{proptangentialcurve}
A curve $c \subset M$ that is always tangential to the leaves, i.e. for all $s$ holds \mbox{$\dot c(s) \in T_{c(s)}N_{c(s)},$} lies in one leaf.
\end{proposition}

With regard to Lemma \ref{lemjacobifield} and Remark \ref{Prop3.1}, we now can define negative homothe\-tical transformations as follows. 

\begin{proposition}\label{prophomotheties}
If $(M, \F)$ is a singular Riemannian folation without horizontally conjugate points, then
\renewcommand{\labelenumi}{(\roman{enumi})}
\begin{enumerate}
\item except at isolated points, $\dim N_{\gamma(t)} = \max \{\dim N_{\gamma(s)}\mid s \in \R\}$ along any horizontal geodesic $\gamma.$
\item The homothetical transformations $h_{\lambda}$ also exist for negative $\lambda$ and  respect the foliation. Particularly, the \emph{reflection} at a leaf respects the foliation.
\end{enumerate}
\end{proposition}

\begin{proof} To prove the first item, let $\gamma(t_0)$ be such that the leaf through $\gamma(t_0)$ is of maximal dimension (along the geodesic); let $N_{\gamma(t_1)}$ be of smaller dimension. By maximality, we can choose $t_1$ in such a way that $\gamma |_{[t_0, t_1)} \subset \Sigma_{\dim(N_{t_0})}.$ Now consider a local tubular neighborhood of $\gamma(t_1) \in N_{t_1}$ of normal radius $\rho$ and its intersection with the stratum $\Sigma_{\dim(N_{t_1})}$. Following \cite{molino}, p.~196, this is a union of geodesic arcs of length $2\rho$ perpendicular to $N_{t_1}.$ As $\gamma$ is perpendicular to $N_{t_1}$ in $\gamma(t_1),$ too, we have
$$\gamma |_{(t_1 - \rho, t_1 + \rho)} \subset \Sigma_{\dim(N_{t_1})} \ \text{or }\; (\gamma |_{(t_1 - \rho, t_1 + \rho)}\setminus \{\gamma(t_1)\}) \cap \Sigma_{\dim(N_{t_1})} = \emptyset,$$
i.e.~the geodesic arc $\gamma |_{(t_1 - \rho, t_1 + \rho)}$ is either contained completely in the stratum $\Sigma_{\dim(N_{t_1})}$ or not at all except the point $\gamma(t_1).$

The first inclusion cannot be fulfilled as $\gamma |_{[t_0, t_1)} \subset \Sigma_{\dim N_{t_0}}$ and we assumed the strata to be different. Because of the existence of positive homothetical transformations, it is only left to show that there is an $\varepsilon < \rho,$ such that $\gamma(t_1 + \varepsilon) \in \Sigma_{\dim (N_{t_0})}.$

Denote the maximal dimension of leaves along $\gamma$ by $k,$ and the dimension of $N_{t_1}$ by $l$. The intersection of a slice through $\gamma(t_1)$ with a plaque of a leaf $N_{t_1 - \varepsilon}$ is $(k-l)$-dimensional and lies within a $(\dim(M)-l)$-sphere around $\gamma(t_1)$. All geodesics beginning in $\gamma(t_1)$ and ending in $N_{t_1 - \varepsilon}$ at the same time as $\gamma$ form a $(k-l)$-dimensional variation of horizontal geodesics which can be continued up to infinity. The negative starting vectors are of course tangential to the negatives of the curves and span a $(k-l)$-dimensional submanifold in $T_{\gamma(t_1)}M.$ Within the tubular neighborhood, we can assume the existence of unique geodesics between arbitrary points, and the negative variation up to $t_1 + \varepsilon$ has to stay $(k-l)$-dimensional. Hence, there is only one leaf meeting the variation at time $t_1 + \varepsilon$ and this is $k$-dimensional which means maximal.

The second item now is easily seen: The proof of (i) shows the existence of the homothetical transformation $h_{-1}$ at $N_{t_1}$. Compose this with positive ones.
\end{proof}

\begin{corollary}\label{corcovering}
Let $N$ be a leaf in $M$ with global tubular neighborhood ${\rm Tub}_\eps(N).$ Then, the homothetical transformations $$h_{\lambda}: S_\delta^N \to S_{\lambda\delta}^N$$ (as long as $\lambda \in (-\eps, \eps)$ and $\delta$ are defined) map leaves diffeomorphically onto each other.
\end{corollary}

\begin{proof} Consider a leaf $L_1.$ By Proposition \ref{prophomotheties}, $h_{\lambda}$ locally maps open sets of $L_1$ diffeomorphically to open sets of other leaves. The image $h_{\lambda}(c)$ of an arbitrary path $c: [0,1] \to L_1$ is tangential to the leaves as a path through the variation of horizontal geodesics defining $h_{\lambda}.$ Hence, $h_{\lambda}$ maps $L_1$ into one leaf $L_2.$ Further, $h_\lambda: L_1 \to L_2$ is injective.

Assume that there is $x \in L_2$ contained in the closure of $h_\lambda(L_1)$ but not in its image. Then there is a variation through horizontal geodesics whose endpoints in $L_2$ converge to $x.$ The limit curve is a horizontal geodesic $\gamma$ beginning in $L_1$ and meeting $N,$ too. Thus, $\gamma(0) \in L_1$ is a preimage of $x$ with respect to $h_{\lambda}$ which means that $h_{\lambda}(L_1)$ is open and closed in $L_2,$ hence equal to $L_2.$ 
\end{proof}

\begin{remark}
In \cite{at} and \cite{lth07} is shown (by completely different methods) that the reflection respects the foliation in general, i.e. without the assumption of no horizontally conjugate points. By Theorem \ref{theotube} we thus generalize: \emph{For a singular Riemannian foliation with properly embedded leaves, the homothetical transformations defined within global tubular neighborhoods are diffeomorphisms between the leaves.}
\end{remark}

\begin{remark}
Due to an essential hint of A. Lytchak, we can specify the last corollary even to the types of leaves, which is stated in the next proposition. Thus, if the foliation has no horizontally conjugate points, a situation as in Lemma \ref{lemacc} cannot occur which means: If there lie specificly regular points on a horizontal geodesic, singular points as well as exceptional points are isolated along this geodesic, and Lemma \ref{lemacctype} does not only hold in $\Sigma_{\rm ex}$ but in $M$ itself.
\end{remark}

\begin{proposition}\label{propiso} If $(M, \F)$ is a singular Riemannian foliation without horizontally conjugate points and with properly embedded leaves, then the homothetical transformations respect the types of leaves, i.e.~they map a leaf $N$ of type $a_N$ diffeomorphically onto a leaf of the same type.
\end{proposition}

\begin{proof}
As the leaves are properly embedded, the foliation is closed. We can then use Theorem 1.7 of \cite{lth07} to see that the quotient $M / \F$ is isometric to a good Riemannian orbifold $N / \Gamma$ without conjugate points. The manifold $N$ and the discrete group $\Gamma$ are constructed in such a way that the stratification by types of leaves of $M$ determines the group $\Gamma$ of isometries of $N$ (for details see \cite{lth07}). Because of our Lemma \ref{lemacctype}, it only remains to check wether the reflection, i.e. the homothetical transformation $h_{-1},$ at a singular leaf respects the types of the neighboring leaves. And this is true on $M$ because it is an obvious fact on $N$ stratified by the action of $\Gamma.$
\end{proof}

We close this subsection by citing a lemma proven by Lytchak and Thorbergsson in \cite{lth07}, Subsection 4.5, which we will need later.

\begin{lemma}\label{lemtangentialjf}
Let $(M, \F)$ be a singular Riemannian foliation, \mbox{$\gamma: (-\infty, \infty) \to M$} a horizontal geodesic. Denote by $W^\gamma$  the space of \emph{horizontal Jacobi fields} along $\gamma,$ i.e.~all Jacobi fields that come from variations of $\gamma$ through horizontal geodesics in which all geodesics meet the same leaves at the same time. For every $t \in (-\infty, \infty),$ $W^\gamma$ then spans the tangential space of the leaf $N_t:$ $$\{J(t) \mid J \in W^\gamma \} = T_{\gamma(t)}N_{t}.$$
In particular, $J(\tilde t) \in T_{\gamma(\tilde t)}N_{\tilde t}$ for all $\tilde t.$
\end{lemma}

\subsection{Singular points and focal points}\label{s:singfocal}

The preceding section already points at the possibility to characterise singular points locally by variations through horizontal geodesics which leads to the notion of focal points. In this chapter, we will now have a deeper view onto the relation of these attributes.

For the following, let $(M, \F)$ be a singular Riemannian foliation without horizontally conjugate points; let $\gamma$ be a horizontal geodesic in $M$ meeting the leaf $N_0$ at time $0.$ Furthermore, consider $N_0$ to be of maximal dimension along $\gamma.$ Then,

\begin{lemma}\label{lemfocalpoint}
$\gamma(t) \text{ is a focal point of }N_{0} \quad \Longleftrightarrow \quad \dim N_t < \dim N_{0}.$
\end{lemma}

\begin{proof}
$\Leftarrow$: This can be seen in the proof of proposition \ref{prophomotheties}: For a leaf $N_{t + \varepsilon}$ near $\gamma(t),$ we find a $(k-l)$-dimensional variation of horizontal geodesics starting in $\gamma(t)$ that is everywhere tangential. A singular point on $\gamma$ thus is a focal point to every leaf of higher dimension met by $\gamma.$

$\Rightarrow$: Let $f(s,\tilde t)$ denote a non-trivial variation of horizontal geodesics starting in $N_{0}$ and meeting the leaf $N_t$ in $\gamma(t).$ The corresponding Jacobi field is tangential in $\gamma(0)$ and in $\gamma(t)$, hence everywhere. Therefore, $f(s,\tilde t)$ is an $N_{\tilde t}$-variation for every leaf met by $\gamma$ at time $\tilde t$ with focal point $\gamma(t),$ if $N_{\tilde t}$ is of higher dimension than $N_t.$ Again, take a tubular neighborhood of $\gamma(t) \in N_t$ so small that there exist unique geodesics between arbitrary points, and let $\eps$ be such that $\gamma(t + \eps)$ is contained in it. Then, $f(s, t + \eps) \in S_{\eps}^{\gamma(t)}$ for all $s,$ where $S_{\eps}^{\gamma(t)}$ denotes the normal sphere of radius $\eps$ around $\gamma(t).$ Assuming $N_{t}$ to have maximal dimension now leads to a contradiction: The leaf $N_{t + \eps}$ must have the same dimension and $$N_{t + \eps} \cap S_{\eps}^{\gamma(t)} = \text{ finitely many points}$$ while the variation is non-trivial and its geodesics are unique, so $\dim N_t < \dim N_{0}.$
\end{proof} 

\begin{remark}
In a singular Riemannian foliation (without additional conditions), singular points are easily seen to be focal points of leaves nearby. But in general, there are focal points in regular leaves, too! A trivial example is the sphere $S^2$ with the standard metric, foliated by points. Geodesics are always horizontal, and conjugate points are horizontally conjugate. Obviously, they are focal points to each other, and regular by definition.
\end{remark}

With regard to the last proof and the proof of Proposition \ref{prophomotheties}, we can improve Lemma \ref{lemfocalpoint} as follows.

\begin{proposition}\label{propsingfocal}
For every singular point $p$ in $M$ and every regular leaf $N$ there is a geodesic  $\gamma: [0,1] \to M$ with $\gamma(0) =: q \in N, \gamma(1) = p,$ that is perpendicular to $N = N_q,$ and $p$ is a focal point of $N_q$ along $\gamma$ of multiplicity
$$\nu(p) = {\rm dim}(N_q)-{\rm dim}(N_p).$$
In particular, the multiplicity of $p$ is independent of the choice of $\gamma.$
\end{proposition}

\begin{proof}
We just drop a perpendicular from $p$ to $N;$ let \mbox{$\gamma: [0,1] \to M$} be a geodesic beginning in $q = \gamma(0)$ perpendicularly to $N = N_q,$ ending in $p.$ As in the proof of Lemma \ref{lemfocalpoint}, there is a local variation of  $\gamma$ of maximal dimension $\nu(p)$ which can be extended up to the leaf $N_q.$ If this variation was not $\nu(p)$-dimensional in $N_q$ the point $q$ would be a focal point along $\gamma$ (namely a focal point of the leaf $N_{t + \eps}$ in the notation used above) which contradics the regularity of $N_q.$
\end{proof}

Combining the latest findings with the general form of Proposition \ref{prophomotheties} part (i) as proven in \cite{lth07} and \cite{at}, yields the main theorem of this chapter, the geometric characterisation of foliations without horizontally conjugate points. 

\begin{theorem}\label{theowhcp}
A singular Riemannian foliation has no horizontally conjugate points if and only if the singular points are exactly the focal points of leaves.
\end{theorem}

To prove the if-part, we will need the following technical lemma that holds under the assumption of singular points being the focal points. 

\begin{lemma}\label{lemdisturb}
Let $\gamma: (- \infty, \infty) \to M$ be a horizontal geodesic. If we have $a < b$ such that $\gamma(b)$ is horizontally conjugate to $\gamma(a)$ along $\gamma,$ for every $\eps > 0$ (sufficiently small) there is a $\delta > 0$ and $c \in (a + \delta, b + \eps)$ such that $\gamma(c)$ is horizontally conjugate to $\gamma(a + \delta).$ 
\end{lemma}

\begin{proof}
We define $\mathcal T := \{J \mid J \text{ is a Jacobi field along }\gamma, J(\tilde t) \in T_{\gamma(\tilde t)}N_{\tilde t}\ \forall_{\tilde t}\}.$ In analogy to \cite{wilking} and \cite{lth}, we can define a bundle $\mathcal V_\gamma$ over $\mathbf R$ by the following extension of $\mathcal T$ at singular points:
$$\mathcal V_t = \{J(t) \mid J \in \mathcal T\} \oplus \{J'(t) \mid J \in \mathcal T, J(t) = 0\}.$$
Then, $\{J'(t) \mid J \in \mathcal T, J(t) = 0\} = 0$ if $t$ is regular by the assumption that singular points are focal points. As we know the latter to be isolated along $\gamma,$ we see that $\mathcal V_\gamma$ is a smooth bundle, and $\dim (\mathcal V_t) = \dim (T_{\gamma(t_0)}N_{t_0})$ for all $t$ wherby $\gamma(t_0)$ is regular (see Proposition 3.2 in \cite{lth}). Define the horizontal part $\mathcal H_t$ to be the orthogonal complement of $\mathcal V_t$ in $T_{\gamma(t)}N_{t}$ and denote the bundle by $\mathcal H_\gamma.$

Wilking showed that there is an associated Morse Sturm differential equation to $\mathcal H_\gamma$ in $\mathbf R^k,$ such that its space of solutions $\mathcal L$ determines just the horizontal Jacobi fields along $\gamma$ which vanish in $a.$ By the generalized index theorem of Morse, the corresponding index form $$I_{a,b}(V, W) =\int_a^b \left< V', W'\right> - \left<\mathcal R(t)V, W \right> dt$$ has $${\rm Ker}(I_{a,b}) = \{J|_{[a,b]} \mid J \in \mathcal L, J(b) = 0\}\quad \text{and}\quad {\rm Ind}(I_{a, b}) = \sum_{t \in (a,b)} \nu(t).$$ 
Here, $V, W$ lie in $H_0^1([a, b], \mathbf R^k),$ the space of absolutely continuous, ${\mathbf R^n}$-valued maps of $[a,b]$ with square integrable derivative and boundary condition $V(a) = V(b) = 0$; $\mathcal R$ is continuous and $\mathcal R (t)$ is self-adjoint for all $t$ (here, we use the notation of Lytchak and Thorbergsson in \cite{lth} which we refer to for details).

Without loss of generality let $\gamma(b)$ be the first conjugate point to $\gamma(a)$. For every $\eps > 0$ sufficiently small, let $\mathcal U \subset H_0^1([a, b+ \eps], \mathbf R^k)$ be a maximal subspace on which the index form is negative definite
$$I_{a , b + \eps}(V,V) < 0\ \text{ for all } V \in \mathcal U.$$
The Morse Sturm differential equation and thus the index form depend continuously on the end points of the interval. Beyond, there is a neighborhood around $N_{\gamma(a)}$ that contains leaves of same or higher dimension only, so by the index theorem the following holds for all $\delta > 0$ sufficiently small: $$I_{a + \delta , b + \eps}(\tilde V,\tilde V) < 0\; \text{ for all } \tilde V \in \tilde{\mathcal U},$$ where $\tilde{\mathcal U} \subset H_0^1([a + \delta, b+ \eps], \mathbf R^k)$ with boundary condition $V(a+ \delta) = V(b + \eps) = 0$ is a subspace of dimension $\dim(U).$ Thus, continuity implies
$$\hspace{2cm}\exists_{c \in (a + \delta, b+ \eps)} : \gamma(c) \text{ is horizontally conjugate to } \gamma(a + \delta).\hspace{2cm}(*)$$ 
\end{proof}

\begin{remark}
We can state $(*)$ more precisely: For $\gamma(a + \delta)$ there is a horizontally conjugate point $\gamma(c)$ with $c \in [b, b + \eps).$

By assumption, there is no horizontally conjugate point between $\gamma(a)$ and $\gamma(b).$ For all $\tilde \eps > 0$ and all $W \in H_0^1([a, b- \tilde \eps], \mathbf R^k),$ we therefore have $I_{a, b - \tilde\eps}(W, W) > 0.$ If $\delta$ in $(*)$ is chosen to be sufficiently small, continuity implies $I_{a + \delta, b - \tilde\eps}(W, W) > 0.$ Hence, $\gamma(a + \delta)$ has no horizontally conjugate point on $\gamma|_{[a + \delta, b)}.$
\end{remark}
\medskip

\noindent \emph{Proof of Theorem \ref{theowhcp}.}
With regard to Lemma \ref{lemfocalpoint} it remains to show that if focal points are singular, then the foliation is without horizontally conjugate points. We prove this by contradiction.

Assume that there are horizontally conjugate points, i.e. there is a horizontal geodesic $\gamma: [0,1] \to M,$ points $t_0, t_1 \in [0,1]$ and a $N_{0}$-Jacobi field $J$ along $\gamma$ such that $J(t_0) \in T_{\gamma(t_0)}N_{t_0},$ but $J(t_1) \notin T_{\gamma(t_1)}N_{t_1}.$ By Lemma \ref{lemtangentialjf}, there exists a $N_0$-Jacobifeld $\tilde J$ along $\gamma$ that is everywhere tangential and $\tilde J (t_0) = J(t_0).$ The difference $J - \tilde J$ is a $N_0$-Jacobi field that vanishes in $t_0$ but not identically. Hence, $\gamma(t_0)$ is a focal point of $N_0.$ 

Wilking shows in \cite{wilking}, Section 3, that then $(J - \tilde J)'(t_0) \in \nu_{\gamma(t_0)}N_{t_0}$, i.e. $J - \tilde J$ is a $N_{t_0}$-Jacobi field. Now, use lemma \ref{lemtangentialjf} again: There is $\hat J$ such that $J - \tilde J - \hat J$ is a non-trivial $N_{t_0}$-Jacobi field vanishing in $0.$ Hence, $\gamma(0)$ is a focal point of $N_0$ and thus singular.

By Lemma \ref{lemdisturb}, we can disturb horizontally conjugate points, so they are not isolated. But for every $\delta > 0$ (notation as in \ref{lemdisturb}), $\gamma(0)$ as well as $\gamma(0 + \delta)$ lie in singular leaves and thus have to be isolated along $\gamma,$ a contradiction!\qed

\begin{corollary}\label{korbolton}
A singular Riemannian foliation with regular leaves of codimension one has no horizontally conjugate points.
\end{corollary}

\begin{proof} Let $N \in \F$ be arbitrary. Bolton shows in \cite{bolton} that the critical points of the normal exponential map $\exp_N^\perp$ are just the singular points if the foliation is of codimension one.
\end{proof}

\section{Tautness of the leaves}\label{s:taut}

\subsection{Taut immersions} \label{ss:immersion}
In order to draw topological conclusions for the leaves from their geometry we will use Morse theory. At first, let us recall some notation. 

A proper embedding $N \to M$ is called \emph{taut} (with respect to 
$\Z_2$), if the energy functional $E_p : P(M, N \times p) \to \R,\ c \mapsto \int_0^1 \|\dot c(t)\|^2 dt,$ is $\Z_2$-perfect for every $p$ in $M,$ that is not a focal point of $N,$ where 
$$P(M, N \times p) = \{c: [0,1] \to M \text{ in } H^1\ \mid\ (c(0), c(1)) \in (N, p)\}$$ 
denotes the Hilbert space of $H^1$-paths starting in $N$ and ending at $p,$ see \cite{grove} and \cite{terngthorb2}. It is well known that the energy functional is a Morse function whose critical points in $P(M, N \times p)$ are exactly the geodesics starting perpendicularly to $N$ and parametrized proportionally to arc length. Now, Morse's index theorem tells us that the index of such a geodesic $\gamma$ is equal to the sum of multiplicities of the focal points along $\gamma.$

If $\mathcal N_p^r =\{c \in P(M, N \times p) \mid E_p(c) \le r\}$ denotes the sublevel set of paths with energy less than or equal to $r \in \R,$ then the number $\mu_k^r$ of critical points of index $k$ of $E_p$ in $\mathcal N_p^r$ is finite. Let  $b_k^r$ be the $k$-th Betti number of $\mathcal N_p^r$ with respect to $\Z_2$ coefficients, so for all $r$ we can state the weak Morse inequalities $$\mu_k^r \ge b_k^r\ \text{ for all } k.$$  
A Morse function is called \emph{$\Z_2$-perfect,} if equality holds for all $r$ and all $k:$ $$\mu_k^r = b_k^r.$$ If $N$ is taut, the focal points of $N$ along a horizontal geodesic determine the $\Z_2$-homology of the space $P(M, N \times p)$ completely.

To show tautness of an immersion, it is a well working strategy to construct linking cycles explicitly, as done by Morse \cite{morse}, p.~244f, and in the famous paper of Bott and Samelson \cite{bottsamelson} (for a good exposition of this theory see \cite{palaisterng}, for the construction of such cycles in special cases see also \cite{sth}, \cite{gth02}, and \cite{thorbdupin}). As we will do so, too, let us shortly review this construction: The Morse inequalities reveal the deep connection between the $k$-th homology group of $\mathcal N_p^r$ and the geodesics of energy $r$ with index $k.$ Given such a geodesic, we need to construct an appropriate, non-trivial class, the linking cycle, in $H_k(N_p^r, \Z_2)$ out of the data of the geodesic, i.e. focal points and their multiplicities, to obtain equality.

As above we define
$$\mathcal N_p^{r-} =\{c \in P(M, N \times p) \mid E_p(c) < r\}.$$
It is a well known fact that $\mathcal N_p^r$ is homotopy equivalent to  $\mathcal N_p^{r-}$ if $r$ is not critical. And if $r$ is critical, there is an $\eps > 0,$ such that $$H_k(\mathcal N_p^{r}, \mathcal N_p^{r-\eps}) = \Z_2^i,$$ where $i$ denotes the number of critical points of $E_p$ with index $k$ and value $r.$ Then the long exact homology sequence tells us that $E_p$ is perfect iff
$$\hspace{5cm}H_k(\mathcal N_p^{r})\to H_k(\mathcal N_p^{r}, \mathcal N_p^{r-\eps})\hspace{3.2cm}(*)$$ is surjective for all $k.$ 

Let $\gamma$ be a critical point of $E_p$ with value $r,$ let $$1 > t_1 > \ldots > t_k > t_{k+1} = 0$$ such that $\{\gamma(t_i)\}_{i=1}^k$ are just the focal points of $N$ along $\gamma.$ Then 
$${\rm ind}(\gamma) = \sum_{i=1}^k \nu(\gamma(t_i)).$$
To show surjectivity in $(*)$ it suffices to construct for every $\gamma$ a compact, ${\rm ind}(\gamma)$-dimensional manifold $K,$ and a smooth map $$\lambda: K \to P(M; N \times p),$$ whereby the following conditions hold:
\begin{enumerate}
\item $\lambda(x_0) = \gamma$ for exactly one $x_0 \in K,$ and there is a neighborhood $U$ of $\gamma$ such that $\lambda|_{\lambda^{-1}(U)}$ is an immersion,
\item $\gamma$ is the only critical point of $E_p$ in the image of $\lambda$, and 
\item $E_p(\lambda(x)) \le E_p(\gamma)=r$ for all $x \in K.$
\end{enumerate}

By the negative gradient flow of $E_p,$ the image $\lambda(K)$ can then be deformed into a smooth cycle $Z_{\gamma}$ in $P(M, N \times p),$ which represents a non-trivial class in $H_{{\rm ind}(\gamma)}(\mathcal N_p^{r}).$ Conditions (i) and (ii) guarantee that the image of this class in $(*)$ is non-trivial, too. So, for every $\gamma$ we would have found a linking cycle $Z_\gamma,$ which shows that $H_k(\mathcal N_p^{r})\to H_k(\mathcal N_p^{r}, \mathcal N_p^{r-\eps})$ is surjective for every $k.$ Thus, tautness of $N$ follows.


\subsection{Topology of singular Riemannian foliations without horizontal conjugate points}\label{ss:topology}

Let $N$ denote a leaf and $P(M, N \times p)$ the corresponding space of paths ending in a point $p$ that is not a focal point of $N$. Critical points of the energy integral $E_p$ are horizontal geodesics from $N$ to $p$ with index equal to the sum of focal points counted with multiplicities. Hence, the easiest case for which we have to construct a linking cycle is a geodesic with only one focal point that is of multiplicity one. If all such critical points of index one are completable, we call $N$ \emph{$0$-taut.} As a first step towards the general construction we show

\begin{theorem}\label{theo0taut}
Let $(M, \mathcal F)$ be a singular Riemannian foliation without horizontal conjugate points on a complete Riemannian manifold $M$. Then, the leaves are $0$-taut.
\end{theorem}

\begin{remark}
The following construction of a linking cycle obviously also works in the situation of a horizontal geodesic with only one focal point whose multiplicity is arbitrary.
\end{remark}

\begin{proof}
At first, let $N$ be a regular leaf and let $\gamma:[0,1] \to M$ be a critical point of $E_p$ as above, i.e. a horizontal geodesic with exactly one focal point $\gamma(t_0)$ of multiplicity one. By Theorem \ref{theowhcp}, $\gamma(t_0)$ is singular. Hence, by Lemma \ref{lemacc} and Lemma \ref{lemacctype}, we can conclude that except isolated points, all points along $\gamma|_{(0,t_0]}$ are of the same type. Furthermore, the isolated points are of higher type.

Choose a global tubular neighborhood ${\rm Tub}_{\eps}(N_{t_0})$ and $\delta < \eps$ such, that $\gamma(t_0 - \delta)$ is specificly regular resp.~of minimal type along $\gamma.$ Then, $N_{\gamma(t_0 - \delta)} \cap S_\eps(\gamma(t_0))$ is a compact, properly embedded submanifold of dimension $\nu(\gamma(t_0)) = 1$ of $N_{\gamma(t_0 - \delta)}.$ If it is not connected, we choose the component -- topologically a circle -- that containts $\gamma(t_0 - \delta)$ and call it $K.$ Now, define 
$$f: K \times [0,t_0] \to M, \text{ with } \left\{\begin{array}{l} f(s,t_0) =  \gamma(t_0) \text{ for all } s \in K\\ \text{and } K = \{f(s, t_0-\delta) \mid s \in K\},\end{array}\right.$$ to be the one-dimensional, smooth variation of horizontal geodesics of the same length as $\gamma|_{[0, t_0]}.$ This variation is unique because the paths $f_{s}|_{[t_0-\delta, t_0]}$ are minimal from $K$ to $\gamma(t_0)$ in $M$ for all $s.$ 

The path through the variation at time $t_0 - \delta$ is tangential to the leaf $N_{\gamma(t_0 - \delta)}.$ Obviously, $\frac{\partial}{\partial s}f(s,t_0)$ is tangential to $N_{t_0}$ for all $s.$ Thus, $$\frac{\partial}{\partial s}f(s,t) \in T_{f_s(t)}N_{\gamma(t)} \text{ for all $s$ and all $t$};$$ and by proposition \ref{proptangentialcurve} we have $f(s, 0) \in N$ for all $s.$ Moreover, the set $$\{f(s,0) \mid s \in K\} \subset N$$ is a compact submanifold (possibly not properly embedded) of $N$ as image of $K:$ 

Cover $\gamma|_{[0,t_0-\delta]}$ by global tubular neighborhoods similar to the covering in the proof of Theorem \ref{theotube}. For this purpose, let the centers be the leaf $N$ as well as all the leaves of higher type than $N_{\gamma(t_0-\delta)}$ and, if these tubular neighborhoods do not yet cover $\gamma|_{[0,t_0-\delta]}$, some specificly regular leaves  (resp.~leaves of the same type as $N_{\gamma(t_0 - \delta)}$). Thus, each point $f(s,t)|_{[0,t_0-\delta]}$ lies within exactly one neighborhood or within the section of two consecutive neighborhoods. Denote the centering leaves by $N = N_0, N_{1}, \ldots, N_{r} \neq N_{\gamma(t_0 - \delta)}$ with $N_{\gamma(t_0 - \delta)} \subset {\rm Tub}(N_r).$ The homotheties defined within the global neighborhoods are diffeomorphisms by Corollary \ref{corcovering}, and with the help of these maps, we move $K$ along the geodesics $f_s(t)$ from $N_{\gamma(t_0 - \delta)}$ to a leaf (of the same type) in the intersection ${\rm Tub}(N_{r}) \cap {\rm Tub}(N_{r-1}),$ then to a leaf in ${\rm Tub}(N_{r-1}) \cap {\rm Tub}(N_{r-2}),$ etc.~up to a leaf $L$ in the intersection ${\rm Tub}(N_{1}) \cap {\rm Tub}(N_0).$ The image of $K$ in $L,$ i.e.~the set $\{f(s,\tilde t) \mid s \in K, \gamma(\tilde t) \in L\},$ is mapped to $\{f(s,0) \mid s \in K\}$ via the orthogonal projection $\pi_N|_{L}: L \to N$ in the last step, hence we loose embeddedness in general.

Now define an injective, smooth map $\lambda: K \to P(M, N \times p)$ by
$$\lambda (s) (t) = \left\{\begin{array}{ll}
f_{s} (t) , & t \in [0 , t_0],\\
 \gamma(t),  & t \in [t_0, 1].
\end{array}\right.$$
The compact image $\lambda(K)$ is an injective, smooth cycle of dimension $\nu(\gamma(t_0))$ in $\mathcal N_p^{E(\gamma)}$ consisting of broken geodesics of same length and thus of same energy. The only non-broken geodesic is $\gamma.$ Hence, the critical point $\gamma$ is completable.

For singular leaves, we can use an idea based on Proposition 2.7 in \cite{terngthorb2}, which in the general case of foliations is due to S.~Wiesendorf, see \cite{wiesend}: By definition, a leaf $N$ is taut, if the energy functional $E_p$ is perfect for $p$ not focal, but by a limiting argument, it suffices to show perfectness for all $p$ that are regular. Thus, along a horizontal geodesic, nearly all points are regular. To show that a singular leaf $N$ is taut, we construct a linking cycle $\lambda(K)$ of a regular leaf $L$ nearby and then add to every broken geodesic in $\lambda(K)$ a minimal geodesic connecting the endpoint of the geodesic in $\lambda(K)$ with $N.$ The corresponding map $$P(M; L \times p) \to P(M; N \times p); c \mapsto \tilde c = \gamma_{c(0)}c$$ is injective and preserves the indices of geodesics as critical points of $E_p.$ Furthermore, the energy levels will just be shifted by a constant. The new set of broken geodesics,  $\widetilde{\lambda(K)},$ by construction can be deformed into a linking cycle of $N.$
\end{proof}

T\"oben has shown in \cite{toeben} that a singular Riemannian foliation with properly embedded leaves and with flat sections in a simply connected symmetric space does never have exceptional leaves. As an application of the previous theorem, we generalize this to

\begin{proposition}\label{propsimplycon}
If $M$ is a complete, simply connected manifold and if $(M, \F)$ is singular Riemannian without horizontal conjugate points and with properly embedded leaves, there are no exceptional leaves in $\mathcal F.$
\end{proposition}

\begin{proof}
Let $N$ be a leaf and $p \in M.$ There is the fibration $$P(M, N \times p) \to N; c \mapsto c(0),$$ where $c: [0,1] \to M$ is a path with $c(0) \in N, c(1) = p.$ The fiber $\Omega_{x_0 p}(M)$ over $x_0 \in N,$ that is the space of paths from $x_0$ to $p,$ is homotopy equivalent to $\Omega_{x_0}(M)$ which is connected as $M$ is simply connected. The long exact sequence of homotopy groups of fibrations yields that $P(M, N \times p)$ is connected:
\begin{center}
\hspace{3mm}\xymatrix@R=6pt{
\ldots \ar[r] & \pi_1(N) \ar[r] & \pi_0(\Omega_{x_0}(M))\ar@{=}[d] \ar[r] & \pi_0(P(M, N \times p)) \ar[r] & \pi_0(N)\ar@{=}[d] \ar[r] & 0\;.\\
&&0&&0}
\end{center}

Assume that there is an exceptional leaf $A \subset M.$ Then let $B$ be a leaf with trivial holonomy meeting a tubular neighborhood of $x \in A$ several times. Under the boundary condition  $B \times x,$ the energy integral thus must have at least two local minima. Now use the minimax principle:

There is a critical point of index one, i.e. a geodesic $c$ from $B$ to $x$ with just one focal point, and this critical point connects two components of $\mathcal B_x^{E(c)^-}$ in $\mathcal B_x^{E(c)}.$ By Theorem \ref{theo0taut}, there is a one-dimensional linking cycle $Z,$ topologically a circle, in $P(M, B \times x)$ with $$E(z) \le E(c)\ \text{ for all } z \in Z$$ and equality holds if and only if $z = c.$ This leads to a contradiction because $Z\setminus \{c\}$ is connected but, as $Z$ is a representative in $H_{{\rm ind}(\gamma)}(\mathcal B_x^{E(c)})$ that maps non-trivially to $H_{{\rm ind}(\gamma)}(\mathcal B_x^{E(c)}, \mathcal B_x^{E(c)^-})$, $Z\setminus \{c\}$ has to meet both components of $\mathcal B_x^{E(c)^-}$ that are only connected by $c.$
\end{proof}


\begin{theorem}\label{theotaut}
The leaves of a singular Riemannian foliation $(M, \mathcal F)$ with embedded leaves and without horizontal conjugate points are taut.
\end{theorem}

\begin{proof}
At first, let $N$ be regular. For each focal point $x$ of $N$ along a horizontal geodesic $\gamma_x,$ we find a corresponding manifold $K_x$ of dimension $\nu(x)$ within a neighboring leaf $L$ that is specificly regular resp.~of minimal type along $\gamma_x$ as in the proof of theorem \ref{theo0taut} and a global variation $f_x(s, t), s \in K_x,$ of geodesics such that $f_x(s, 0) \in N, f_x(s,1) = x$ for all $s \in K_x$ because of Proposition \ref{propsingfocal}. Proposition \ref{propiso} tells us that leaves of different dimension or different type than $L$ are isolated along $\gamma_x$ and therefore, we can carry over the whole proof of the previous theorem: Compactness of the image of $K_x$ in $N$ follows analogously as long as we choose all these singular and exceptional leaves of higher type along $\gamma_x$ to be centers of tubular neighborhoods of the covering. Hence, the only work to do regarding our Theorem \ref{theotaut} is an iteration of this idea. As in \cite{bottsamelson}, we can use a fiber bundle construction.

The notation is the same as above. Denote the focal points on $\gamma$ by $$\{\gamma(t_j)\}_{j=1}^k, 1 = t_0 > t_1 > \ldots > t_k > t_{k+1} = 0, \text{ with } {\rm ind}(\gamma) = \sum_{j=1}^k \nu(\gamma(t_j)).$$

For all $j = 1, \ldots, k$ let $f_{j}(s, \cdot): [0, t_j] \to M$ denote the $\nu(\gamma(t_j))$-dimensional variation of geodesics from $N$ to $\gamma(t_j)$ parametrized by $K_j$ (similar to the proof of Theorem \ref{theo0taut}). This variation consists of all minimal geodesics between $K_j$ and $\gamma(t_j)$ continued up to the length $L(\gamma|_{[0, t_j]}).$ In this situation holds

\begin{lemma}\label{lemfixedfocal}
If $\gamma(t_i)$ is a focal point of $N$ along $\gamma,$ then $f_{j}(s,t_i)$ is a focal point of $N$ of same multiplicity along $f_{j}(s, \cdot)$ for $j = 1, \ldots, i-1$ and for all $s \in K_j.$ Moreover, we have $$N_{f_{j}(s,t)} = N_{\gamma(t)} \text{ for all } j = 1, \ldots, k+1 \text{ and }s \in K_j, t \in [0, t_j].$$
\end{lemma}

\noindent\begin{minipage}{0.5\textwidth}
\begin{center}
\begin{overpic}[width=6cm
]{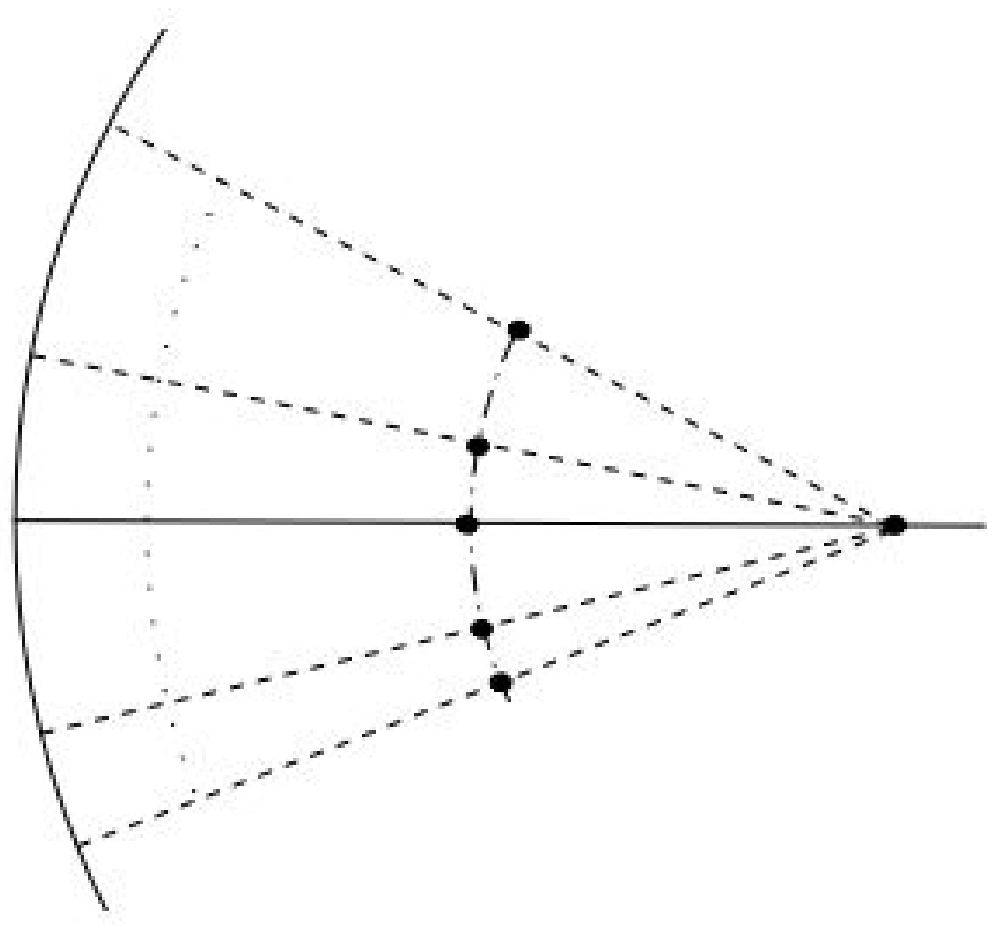}
\put(97,40){$\gamma$}\put(86,37){$\gamma(t_j)$}\put(17,60){Variation}\put(14,55){$f_j: K_j \times [0, t_j]$}\put(17,85){$N$}\put(48,37){$\gamma(t_i)$}\put(4,37){$\gamma(0)$}\put(52,62){$f_j(s, t_i) \in N_{\gamma(t_i)}$}
\end{overpic} 
\end{center}
\end{minipage}
\begin{minipage}{0.5\textwidth}
\begin{proof}
This results directly from Lemma \ref{lemjacobifield} and Proposition \ref{propsingfocal} as well as from the variation formula of the energy.
\end{proof}
\begin{remark} As remarked in \cite{toeben}, focal points along geodesics within such a variation are not fixed in general (i.e. in an arbitrary singular Riemannian foliation), neither are their multiplicities. Although easily proven with the help of Section \ref{s:whcp} here, the lemma is essential to guarantee that the iterated bundle will be a manifold of dimension ${\rm ind}(\gamma)$ as we will see now. 
\end{remark}
\vspace{0.4cm}
\end{minipage}

\noindent Knowing that along all geodesics of the variations $f_{j},$ there are the same focal points at the same time as along $\gamma$ we can define
$$K := \{(s_1, \ldots, s_k) \mid s_1 \in K_1(\gamma(1)), s_2 \in K_2(s_1), \ldots, s_k \in K_k(s_{k-1})\},$$
where $K_i(s_{i-1})$ denotes the manifold of dimension $\nu(\gamma(t_i))$ which focalizes to the point $t_i$ along the geodesic $f_{s_{i-1}}(t), t \in [0, t_{i-1}],$ for all $i.$ And by the Lemma \ref{lemfixedfocal}, $K$ is a compact manifold, a fiber bundle, of dimension ${\rm ind}(\gamma).$

Now define the map $\lambda: K \to P(M, N \times p)$ as follows:
$$\lambda (s_1, \ldots, s_k) (t) = \left\{\begin{array}{ll}
f_{s_i}(t), & t \in [t_i, t_{i+1}], i = 1, \ldots, k\\
\gamma|_{[t_1,1]}(t), & t \in [t_1, 1].
\end{array}\right.$$
The cycle $\lambda(K)$ thus consists of broken geodesics, all of same length; the only one that is non-broken and therefore the critical point of $E_p$ is $\gamma$ itself.

\begin{minipage}[t]{0.9\textwidth}
\emph{Sketch of the construction of $\lambda(K)$ in $M$.}
\begin{center}
\begin{overpic}[width=6cm
]{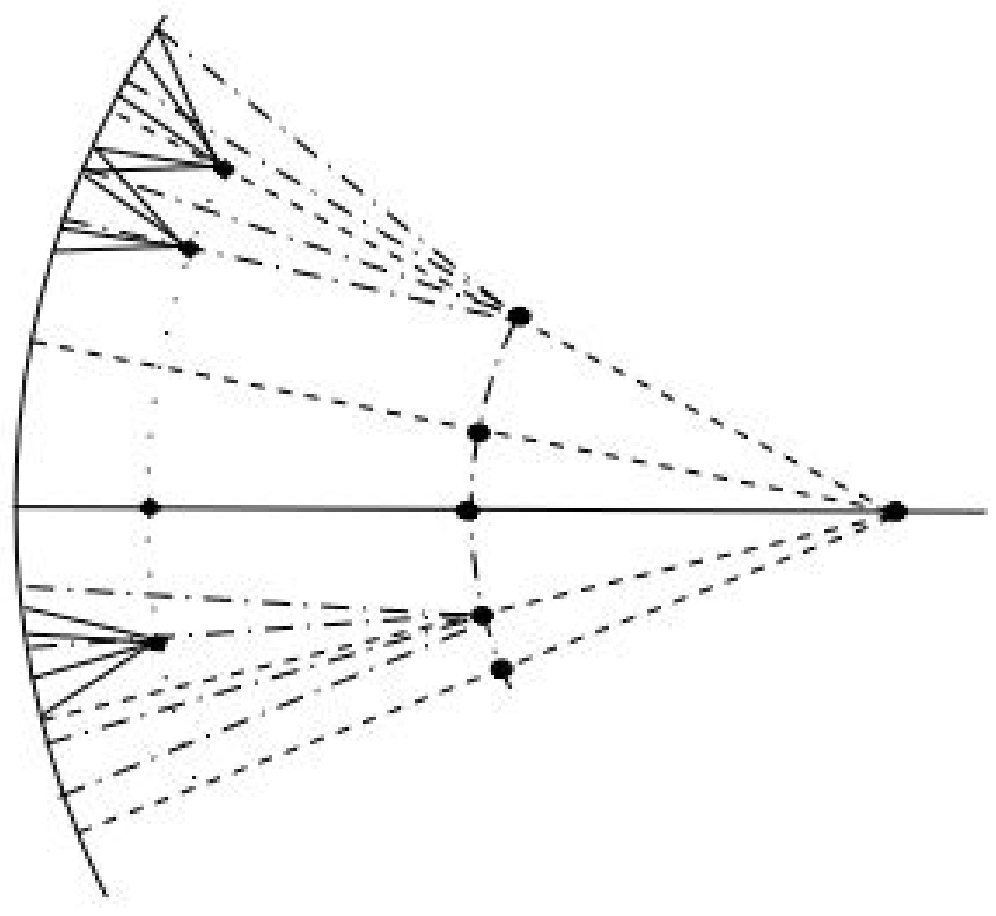}
\put(97,39){$\gamma$}\put(18,88){$N$}\put(12,61){etc.}
\end{overpic} 
\end{center}
\vspace{5mm}
\end{minipage}

Furthermore, $\lambda(K)$ has dimension ${\rm ind}(\gamma)$ and is injective in the path space because each step of the iteration, i.e. of the bundle construction, yields only different broken geodesics. Hence, all requirements for a linking cycle for $\gamma$ are fullfilled. By arbitraryness of $\gamma,$ the claim follows for regular leaves $N.$

For a singular leaf $N$ on $\gamma$, the argumentation is the same as above in the proof of Theorem \ref{theo0taut}: Assume $p$ to be regular and construct the linking cycle of a regular leaf $L$ near $N$. Then, add the geodesics connecting the endpoints of the cycle in $L$ with $N$. The new manifold of broken geodesics fulfills all conditions for a linking cycle for $\gamma.$ 

\end{proof}

We will at last give an example in order to illustrate the cycle construction. 

\begin{minipage}[h]{0.5\textwidth}
\begin{example} Assume that the leaf $N$ in $M = \R^3$ is a $2$-torus, and the critical point in $P(M, N \times p)$ is a horizontal geodesic $\gamma$ as in the figure on the right. The index of $\gamma$ is two, because there are the two focal points $p_1$ and $p_2$ each of them with multiplicity one.
\end{example}
\end{minipage}
\begin{minipage}{0.5\textwidth}
\begin{center}
\begin{overpic}[width=4.3cm
]{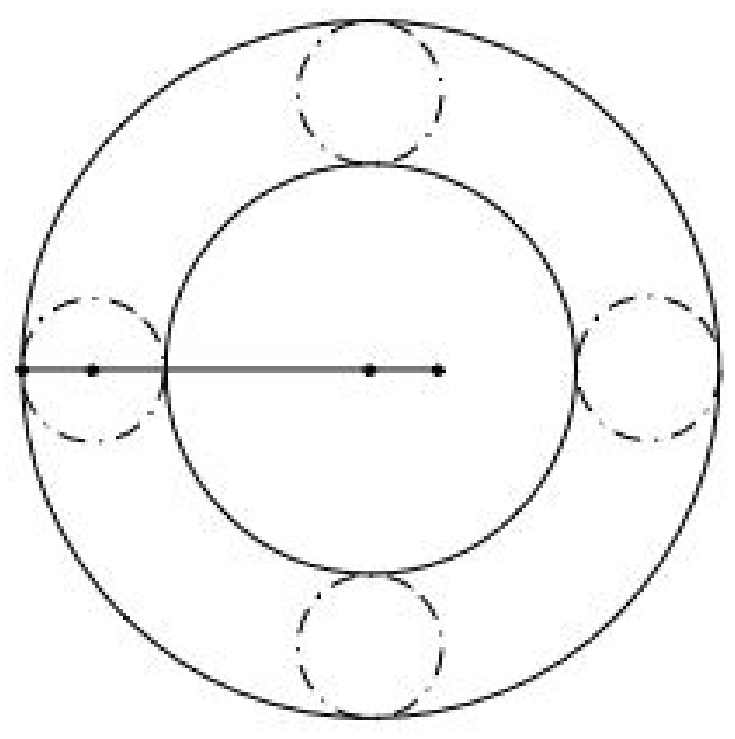}
\put(10,80){$N$}\put(35,45){$\gamma$}\put(49,46){$p_1$}\put(14,46){$p_2$}\put(59,46){$p$}
\end{overpic} 
\end{center}
\end{minipage}

\begin{minipage}[h]{0.5\textwidth}
\begin{center}
\begin{overpic}[width=4.3cm
]{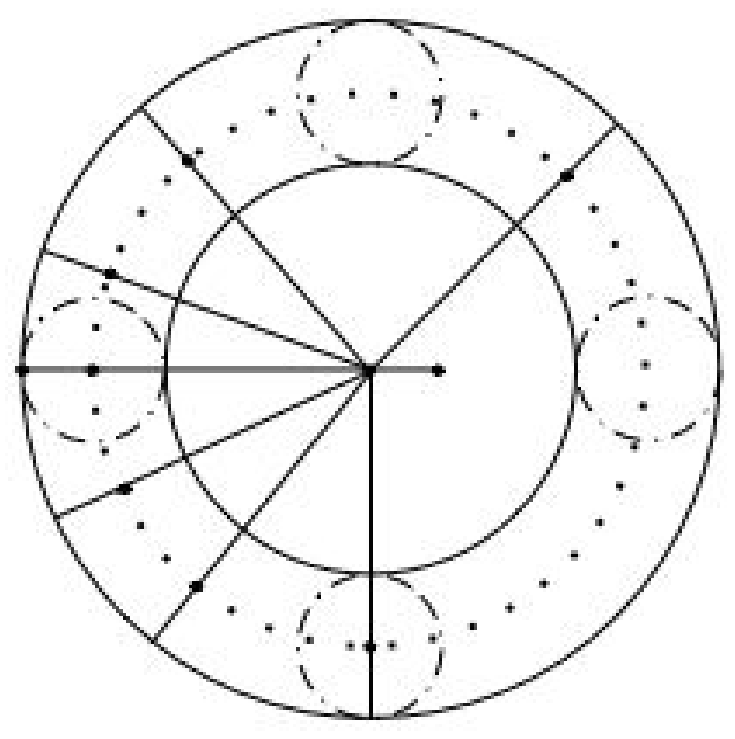}
\put(10,80){$N$}\put(35,46){$\gamma$}\put(51,46){$p_1$}\put(14,46){$p_2$}\put(59,46){$p$}\put(12,70){$K_1$}\put(43,60){\ldots $f_s$\ldots}\put(10,37){$K_2(s_0)$}\put(0,48){$s_0$}
\end{overpic} 
\end{center}
\end{minipage}
\begin{minipage}{0.5\textwidth}
Let $K_2(s), s \in K,$ be the appropriate, one-dimensional focal manifold, which can be identified as meridian through $s$ of the torus.

Thus, the linking cycle of $\gamma$ in $P(M, N \times p)$ is topologically the torus and consists of broken geodesics as outlined in this third figure.
\end{minipage}

\begin{minipage}[h]{0.5\textwidth}
The focal manifold $K_1$ can be identified with the outer latitude of the torus. There is an one-dimensional variation $f_s(t), s \in K_1,$ through geodesics starting perpendicularly in $s \in K_1 \subset N$ and ending in the focal point $p_1.$ On each of the geodesics lies exactly one focal point like $p_2$ on $\gamma.$
\end{minipage}
\begin{minipage}{0.5\textwidth}
\begin{center}
\begin{overpic}[width=4.3cm
]{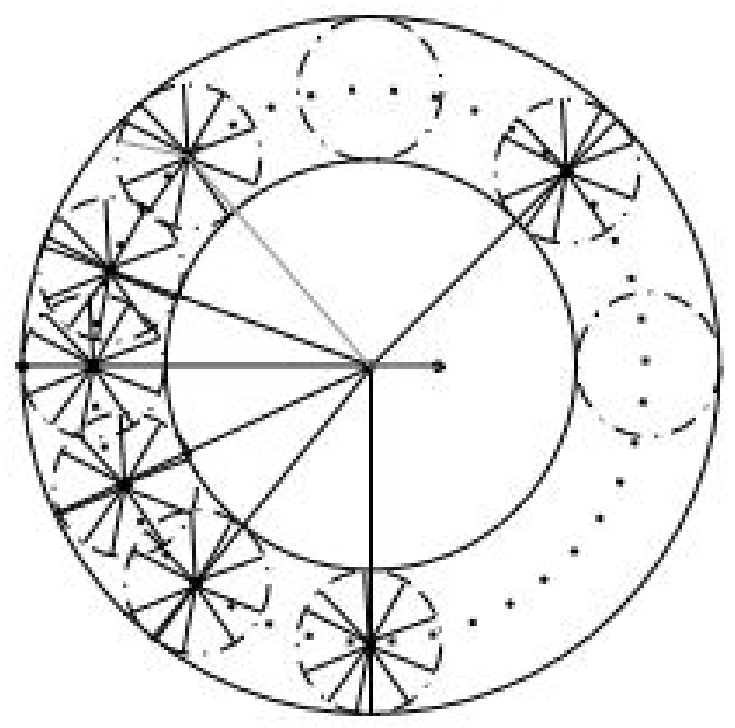}
\put(10,80){$N$}\put(43,-5){figure 3}\put(35,46){$\gamma$}\put(59,46){$p$}
\end{overpic} 
\end{center}
\end{minipage}

\end{document}